\documentclass[final,1p,times]{elsarticle}

\usepackage{url}
\usepackage{amsmath, epsfig, cite}
\usepackage{amssymb}
\usepackage{amsfonts}
\usepackage{latexsym}
\usepackage{graphicx}
\usepackage{amsthm,array}
\usepackage{color,xcolor}
\usepackage{algorithm}
\usepackage{algorithmic,float}
\usepackage{setspace}

\usepackage{array}
\newcommand{\bN} { {\mathbb{N}}}

\newcommand{\bQ} { {\mathbb{Q}}}
\newcommand{\bZ} { {\mathbb{Z}}}

\newcommand{\bF} { {\mathbb{F}}}

\DeclareMathOperator{\Spr}{Spr}

\def\id{\operatorname{id}}
\def\lc{\operatorname{lc}}

\def\lt{\operatorname{lt}}
\def\diag{\operatorname{diag}}
\def\Spr{\operatorname{Spr}}
\def\Const{\operatorname{Const}}
\def\supp{\operatorname{Supp}}
\def\Sol{\operatorname{Sol}}
\newtheorem{thm}{Theorem}[section]

\newtheorem{cor}[thm]{Corollary}
\newtheorem{lem}[thm]{Lemma}
\newtheorem{prop}[thm]{Proposition}

\newtheorem{exam}[thm]{Example}

\newtheorem{problem}[thm]{Problem}

\makeatletter
\newenvironment{breakablealgorithm}
  {
   \begin{center}
     \refstepcounter{algorithm}
     \hrule height.8pt depth0pt \kern2pt
     \renewcommand{\caption}[2][\relax]{
       {\raggedright\textbf{\ALG@name~\thealgorithm} ##2\par}%
       \ifx\relax##1\relax 
         \addcontentsline{loa}{algorithm}{\protect\numberline{\thealgorithm}##2}%
       \else 
         \addcontentsline{loa}{algorithm}{\protect\numberline{\thealgorithm}##1}%
       \fi
       \kern2pt\hrule\kern2pt
     }
  }{
     \kern2pt\hrule\relax
   \end{center}
  }

  \def\eatspace#1{#1}
\def\step#1#2{\par\kern1pt\hangindent#2em\hangafter=1\noindent\rlap{\small#1}\kern#2em\relax\eatspace}

\makeatother

\begin{document}
\begin{frontmatter}



\title{Solutions to the First Order Difference Equations in the Multivariate Difference Field}
\author[1]{Lixin Du}\ead{lixindumath@gmail.com}
\author[2]{Yarong Wei\corref{cor1}}\ead{yarongwei@email.tjut.edu.cn}
\address[1]{Institute for Algebra, Johannes Kepler University, Linz, A4040, Austria}
\address[2]{School of Science, Tianjin University of Technology, Tianjin 300384, PR China}
\cortext[cor1]{Corresponding author}

\begin{abstract}
The bivariate difference field provides an algebraic framework for a sequence satisfying a recurrence of order two. Based on this, we focus on sequences satisfying a recurrence of higher order, and consider the multivariate difference field, in which the summation problem could be transformed into solving the first order difference equations.
We then show a criterion for deciding whether the difference equation has a rational solution and present an algorithm for computing one rational solution of such a difference equation, if it exists. Moreover we get the rational solution set of such an equation.
\end{abstract}



\begin{keyword}
Difference Equation \sep Rational Solutions \sep Multivariate Difference Field \sep Recurrence Relation



\end{keyword}

\end{frontmatter}


\section{Introduction}
Karr developed an algorithm for indefinite summation \citep{Karr-1981,Karr-1985} based on the theory of difference fields. He introduced so called $\Pi\Sigma-$fields, in which a sequence satisfying a recurrence of order one can be described. Inspired by Karr's algorithm, Hou and Wei \citep{Hou-Wei-2023} introduced so called bivariate difference fields, in dealing with the summation problem involving a sequence satisfying a recurrence of order two. They transformed the summation problem into solving the first order difference equations in the bivariate difference field, and gave an algorithm to find a rational solution of the bivariate difference field.

Based on this, we focus on sequences satisfying a recurrence of higher order. Let $(\bF,\varphi)$ be a difference field, i.e., $\varphi$ is an automorphism of a field $\mathbb{F}$. We define a \emph{multivariate difference field} extension $(\bF(\alpha_1,\alpha_2,\ldots,\alpha_n),\varphi)$  of $(\mathbb{F},\varphi)$ to be a field extension with
\begin{enumerate}
\item $\alpha_1,\alpha_2,\ldots,\alpha_n$ being algebraically independent over $\mathbb{F}$.
\item $\varphi$ being a field automorphism of $\bF(\alpha_1,\alpha_2,\ldots,\alpha_n)$ and
\[
(\varphi(\alpha_1,\alpha_2,\ldots,\alpha_n))=(\alpha_1,\alpha_2,\ldots,\alpha_n)A,
\]
where $A\in\bF^{n\times n}$ is a diagonalizable matrix over $\bF$.
\end{enumerate}
Then the summation problem involving sequences satisfying a higher order recurrence could be transformed into solving the difference equations in the multivariate difference field.

For example, let $T_n$ be the Tribonacci sequence  which satisfies
\[ T_{n+3} = T_n+T_{n+1}+T_{n+2}, \quad T_0=0,\quad T_1=T_2=1. \]
We consider the difference field $(\mathbb{C}(\alpha_1, \alpha_2,\alpha_3), \sigma)$ with
\[
\varphi|_{\mathbb{C}} = {\rm id}, \quad
(\varphi(\alpha_1,\alpha_2,\alpha_3))=(\alpha_1,\alpha_2,\alpha_3)\left(
  \begin{array}{ccccc}
    0 & 0 &  1 \\
    1 & 0 &  1 \\
    0 & 1 & 1
  \end{array}
\right).
\]
Then $T_n$,$T_{n+1}$,$T_{n+2}$ are represented by $\alpha_1,\alpha_2,\alpha_3$. To find a closed form of the summation $\sum_{n=1}^mT_n$, we try to search a sequence $G_n$ such that $G_{n+1}-G_n=T_n$ for $n\geq1$. This can be transformed into find a solution $g\in\mathbb{C}(\alpha_1, \alpha_2,\alpha_3)$ of the following difference equation
 \[
 \varphi(g(\alpha_1, \alpha_2,\alpha_3))-g(\alpha_1, \alpha_2,\alpha_3)=\alpha_1.
 \]
By using the algorithm in this paper, one can find
\[\varphi\left(\frac{1}{2}\left(-\alpha_1+\alpha_3\right)\right)-\left(\frac{1}{2}\left(-\alpha_1+\alpha_3\right)\right)=\alpha_1.\]
So $G_n=\frac{1}{2}(-T_n+T_{n+2})=\frac{1}{2}(T_{n-1}+T_{n+1})$ for $n\geq1$.
This implies the following identity given in \citep[Theorem 2]{Kilic-2008}:
\[\sum_{n=1}^mT_n=\frac{1}{2}(T_m+T_{m+2}-1).\]

%
Over the past two decades, many efficient algorithms have been developed for summation of various sequences, like hypergeometric terms \citep{Gosper-1978,Zeilberger-1990,Zeilberger-1991}, $q$-hypergeometric terms \citep{Chen-2005,Du-2016,Graev-1992} and $P$-recursive sequences \citep{Abramov-1999}. More results are listed in \citep{Chen-2017}. There were several extensions of Karr's summation algorithm \citep{Karr-1981,Karr-1985} about difference-field-based techniques, and these extensions handled sums where sequences can appear in the denominator. Schneider\citep{Schneider-2001} first observed that Karr's summation algorithm can be used to solve the creative telescoping problem, and he developed Karr's summation algorithm in \citep{Schneider-2008,Schneider-2013}. Bronstein \citep{Bronstein-2000} described the solutions of some difference equations from the point of view of differential fields. Abramov, Bronstein, Petkov\v{s}ek and Schneider \citep{Abramov-2021} gave an algorithm for computing hypergeometric solutions and rational solutions of some difference equations in the so called $\Pi\Sigma^*$-fields.

We focus on the following two multivariate difference fields in this paper.
\begin{enumerate}
\item The multivariate difference fields  $(\bF(\alpha_1,\alpha_2,\ldots,\alpha_n),\varphi)$ with
\begin{align}\label{eq-1-0}
\varphi|_{\bF}=\text{id},\quad(\varphi(\alpha_1,\alpha_2,\ldots,\alpha_n))=(\alpha_1,\alpha_2,\ldots,\alpha_n)\left(
  \begin{array}{ccccc}
    0 & 0 & \cdots & 0 & u_1 \\
    1 & 0 & \cdots & 0 & u_2 \\
    0 & 1 & \cdots & 0 & u_3 \\
\vdots &\vdots &\ddots&\vdots&\vdots\\
    0 & 0 & \cdots & 1 & u_n \\
  \end{array}
\right)
\end{align}
and $u_1\in\bF\setminus\{0\}$, $u_i\in\bF$ for $2\leq i\leq n$, where the above matrix is a diagonalizable matrix over $\bF$ with nonzero eigenvalues.
\item The multivariate difference fields   $(\bF(\alpha_1,\alpha_2,\ldots,\alpha_n),\varphi)$ with
\begin{align}\label{eq-1-1}
\varphi|_{\bF}=\text{id},\quad(\varphi(\alpha_1,\alpha_2,\ldots,\alpha_n))=(\alpha_1,\alpha_2,\ldots,\alpha_n)\left(
  \begin{array}{ccccc}
    \lambda_1 & 0 & \cdots &  0 \\
    0 & \lambda_2 & \cdots & 0 \\
\vdots &\vdots &\ddots&\vdots\\
    0 & 0 & \cdots &\lambda_n \\
  \end{array}
\right)
\end{align}
and $\lambda_i\in\bF\setminus\{0\}$ for $1\leq i\leq n$.
 \end{enumerate}
Let $(\bF(\alpha_1,\alpha_2,\ldots,\alpha_n),\varphi)$ be a multivariate difference field. Given a nonzero $c\in\bF$ and a rational function $f\in\bF(\alpha_1,\alpha_2,\ldots,\alpha_n)$, we shall find all rational solutions $g$ in $\bF(\alpha_1,\alpha_2,\ldots,\alpha_n)$ of the difference equation
\begin{align}\label{eq-1-2}
c\varphi(g)-g=f.
\end{align}

One feature of the above consideration is that the constant field $\Const_{\varphi}\bF(\alpha_1,\alpha_2,\ldots,\alpha_n)=\{f\in\bF(\alpha_1,\alpha_2,\ldots,\alpha_n)|\varphi
(f)=f\}$ of the multivariate difference field $(\bF(\alpha_1,\alpha_2,\ldots,\alpha_n),\varphi)$ could be larger than $\bF$. Instead of finding a degree bound for polynomial solutions and a universal denominator for rational solutions (which may not exist), we find the structure of the rational solution set of the difference equation \eqref{eq-1-2}. This extends the results in the bivariate difference field \citep{Hou-Wei-2023}.
In particular, when $\lambda_1=\lambda_2=\cdots=\lambda_n$ in \eqref{eq-1-1}, solving the difference equation \eqref{eq-1-2} could be regarded as a generalized for the $q$-summability problem of rational functions \citep{Chen-Singer-2014}.

The outline of this article is as follows.
In Section \ref{sec-2}, we character the constant field and present an algorithm for computing the spread set of any two polynomials in the multivariate difference field.
In Section \ref{sec-3}, we show a necessary and sufficient condition on the existence of rational solutions of the difference equation \eqref{eq-1-2},
and present an algorithm \footnote{\texttt{The SageMath package is available in \url{https://github.com/LixinDu/difference_field}}.} for computing one rational solution of such an equation, if it exists.
In Section \ref{sec-4}, we show a relationship between the rational solution set of the difference equation and the constant field, so that the rational solution set can be constructed.

\section{Multivariate Difference Field}\label{sec-2}
Let $\overline{\bQ}$ be the algebraic closure of $\bQ$, and $\bF$ be a subfield of $\overline{\bQ}$.
Let $\bF(\alpha_1,\alpha_2,\ldots,\alpha_n)$ be the field of rational functions in $n$ variables $\alpha_1,\alpha_2,\ldots,\alpha_n$, and  $(\bF(\alpha_1,\alpha_2,\ldots,\alpha_n),\varphi)$ be a multivariate difference field extension of the difference field $(\bF,\varphi)$.
Let $\bF[\alpha_1,\alpha_2,\ldots,\alpha_n]$ denote the set of polynomials in $n$ variables $\alpha_1,\alpha_2,\ldots,\alpha_n$. For a polynomial $f\in\bF[\alpha_1,\ldots,\alpha_n]$, let $\deg f$ denote the \emph{total degree} of variables $\alpha_1,\ldots,\alpha_n$, and $\deg_{\alpha_i}f$ denote the degree of $f$ in the variable $\alpha_i$ for $i=1,2,\ldots, n$. Let $\lt(f)$ and $\lc(f)$ denote the leading term and the coefficient of the leading term of $f$ under a fixed monomial order, respectively. A polynomial $f$ is said to be \emph{monic} if $\lc(f)=1$. We use the graded lexicographical ordering with $\alpha_1\succ\cdots\succ\alpha_n$ in this paper.

\subsection{A difference isomorphism}

Recall that two difference fields $(\bF_1,\varphi_1)$ and $(\bF_2,\varphi_2)$ are \emph{isomorphic}, if there exists a field isomorphism $\tau:\bF_2\rightarrow\bF_1$ such that
$\varphi_1\tau=\tau\varphi_2$.

\begin{thm}\label{thm-2-1}
Let $(\bF(\alpha_1,\alpha_2,\ldots,\alpha_n),\varphi)$ be a multivariate difference field with $\varphi|_{\bF}=\id$ and $(\varphi(\alpha_1,\alpha_2,\ldots,\alpha_n))=(\alpha_1,\alpha_2,\ldots,\alpha_n)A$, where $A$ is a diagonalizable matrix over $\bF$ with nonzero eigenvalues $\lambda_1$, $\lambda_2$, $\ldots$, $\lambda_n$.
Then
the multivariate difference field $(\bF(\alpha_1,\alpha_2,\ldots,\alpha_n),\varphi)$ and the multivariate difference field $(\bF(\alpha_1,\alpha_2,\ldots,\alpha_n),\sigma)$ are isomorphic, where
\begin{align*}
\sigma|_{\bF}=\id \ {\rm and }\ (\sigma(\alpha_1,\alpha_2,\ldots,\alpha_n))
=(\alpha_1,\alpha_2,\ldots,\alpha_n)\diag(\lambda_1,\lambda_2,\ldots,\lambda_n).
\end{align*}
\end{thm}
\begin{proof}
Let $X_1, X_2,\ldots, X_n$ be the eigenvectors corresponding to the eigenvalues $\lambda_1$, $\lambda_2$, $\ldots$, $\lambda_n$ of matrix $A$.
Define a field isomorphism $\tau:\bF(\alpha_1,\alpha_2,\ldots,\alpha_n)\rightarrow\bF(\alpha_1,\alpha_2,\ldots,\alpha_n)$ with
\[\tau|_{\bF}=\id \ {\rm and}\ (\tau(\alpha_1,\alpha_2,\ldots,\alpha_n))=(\alpha_1,\alpha_2,\ldots,\alpha_n)(X_1,X_2,\ldots,X_n).\]
Then
\[(X_1,X_2,\ldots,X_n)^{-1}A(X_1,X_2,\ldots,X_n)=\diag(\lambda_1,\lambda_2,\ldots,\lambda_n),\]
which implies
\[\varphi\tau=\tau\sigma \ {\rm and}\  \tau^{-1}\varphi=\sigma\tau^{-1}.\]
\end{proof}

\begin{cor} Let
\[
A=\{g\in\bF(\alpha_1,\alpha_n,\ldots,\alpha_n)|c\varphi(g)-g=f\}
\]
and
\[B=\{g\in\bF(\alpha_1,\alpha_n,\ldots,\alpha_n)|c\sigma(g)-g=\tau^{-1}(f)\},
\]
where $c\in\bF\setminus\{0\}$, $f\in\bF(\alpha_1,\alpha_n,\ldots,\alpha_n)$ and $\varphi$, $\sigma$ are defined in Theorem \ref{thm-2-1}. Then there is a one to one correspondence between sets $A$ and $B$.
\end{cor}
\begin{proof}
The correspondence is given by the map $g\mapsto\tau^{-1}(g)$. Indeed, by the definition of $\varphi$ and $\sigma$, we have
\begin{align*}
c\varphi(g)-g=f \Leftrightarrow
\tau^{-1}(c\varphi(g))-\tau^{-1}(g)=\tau^{-1}(f)\Leftrightarrow
c(\tau^{-1}\varphi)(g)-\tau^{-1}(g)=\tau^{-1}(f)\\
\Leftrightarrow
c(\sigma\tau^{-1})(g)-\tau^{-1}(g)=\tau^{-1}(f)
\Leftrightarrow
c\sigma(\tau^{-1})(g)-\tau^{-1}(g)=\tau^{-1}(f).
\end{align*}
\end{proof}
From now on, we only consider the following problem.
\begin{problem}
Let $(\bF(\alpha_1,\alpha_2,\ldots,\alpha_n),\sigma)$ be the multivariate difference field, where
\begin{align}\label{eq-2-0}
\sigma|_{\bF}={\rm id},\ {\rm and}\  (\sigma(\alpha_1,\alpha_2,\ldots,\alpha_n))=(\alpha_1,\alpha_2,\ldots,\alpha_n)\diag(\lambda_1,\lambda_2,\ldots,\lambda_n)
\end{align}
with $\lambda_i\in\bF\setminus\{0\}$. Given $c\in\bF\setminus\{0\}$ and $f\in\bF(\alpha_1,\alpha_2,\ldots,\alpha_n)$, we want to find all rational solutions $g\in\bF(\alpha_1,\alpha_2,\ldots,\alpha_n)$ such that
\begin{align}\label{eq-2-10}
c\sigma(g)-g=f.
\end{align}
\end{problem}

Note that, solving the difference equation \eqref{eq-2-10} with $\lambda_i=\lambda_j (1\leq i\neq j\leq n)$ could be regarded as the $q$-summability problem of rational functions in multivariables.
The criterion of $q$-summability for a rational function in one variable was firstly given by Chen and Singer \citep{Chen-Singer-2012}. For $q$-summability for a bivariate rational function, Chen and Singer show a criterion \citep{Chen-Singer-2014}, and Wang presents an algorithmic proof \citep{Wang-2021}.

\subsection{The constant field}
In a multivariate difference field $(\bF(\alpha_1,\alpha_2,\ldots,\alpha_n),\sigma)$,  we define the following set
\[\Const_{\sigma}\bF(\alpha_1,\alpha_2,\ldots,\alpha_n)=
\{f\in(\bF(\alpha_1,\alpha_2,\ldots,\alpha_n)|\sigma (f)=f\},\]
which is a subfield of $\bF(\alpha_1,\alpha_2,\ldots,\alpha_n)$. We call it the \emph{constant field} of  $(\bF(\alpha_1,\alpha_2,\ldots,\alpha_n),\sigma)$.

Let ${\boldsymbol{\lambda}}=(\lambda_1,\lambda_2,\ldots,\lambda_n)\in\overline{\bQ}^n$.
The \emph{exponent lattice} of these algebraic numbers is defined by
\[
{U_{\boldsymbol{\lambda}}}=\{(i_1,i_2\ldots,i_n)\in\bZ^n|\lambda_1^{i_1}\lambda_2^{i_2}\ldots\lambda_n^{i_n}=1\}.
\]
Then $U_{\boldsymbol{\lambda}}$ is a free $\bZ$-module, and ${U_{\boldsymbol{\lambda}}}$ admits a finite basis.
Ge gave an algorithm for computing a basis of ${U_{\boldsymbol{\lambda}}}$ in his PhD thesis \citep{Ge-1993}, and there are several variants of Ge's algorithm\citep{Manuel-2005,Manuel-2023,Paolo-2014,Sturmfels-1995,Zheng-2020,Zheng-2022,Zheng-2019}.

Now we character the constant field of the multivariate difference field $(\bF(\alpha_1,\alpha_2,\ldots,\alpha_n),\sigma)$ using the exponent lattice of $\lambda_1,\lambda_2,\ldots,\lambda_n$, where $\sigma$ is defined by \eqref{eq-2-0}.

\begin{thm}
Let $(\bF(\alpha_1,\alpha_2,\ldots,\alpha_n),\sigma)$ be the multivariate difference field, which is defined by \eqref{eq-2-0}. Let $r\in\bN$ be the rank of the exponent lattice $U_{\boldsymbol{\lambda}}$.

\begin{enumerate}
\item If $r=0$, then $\Const_{\sigma}\bF(\alpha_1,\alpha_2,\ldots,\alpha_n)=\bF$.
\item If $r>0$, let
$\{\bf{a_1},\bf{a_2},\ldots,\bf{a_r}\}$ be a basis of $U_{\boldsymbol{\lambda}}$ with
${\bf{a_i}}=(a_{i,1}, a_{i,2}, \ldots, a_{i,n})\in\bZ^n$, i.e.,
\begin{align}\label{eq-2-9}
U_{\boldsymbol{\lambda}}=\bZ{\bf{a_1}}\oplus\bZ{\bf{a_2}}\oplus\cdots\oplus\bZ{\bf{a_r}}.
\end{align}
Then
\[\Const_{\sigma}\bF(\alpha_1,\alpha_2,\ldots,\alpha_n)=\bF({\boldsymbol{\alpha}}^{\bf{a_1}},{\boldsymbol{\alpha}}^{\bf{a_2}},\ldots,{\boldsymbol{\alpha}}^{\bf{a_r}}),\]
where ${\boldsymbol{\alpha}}^{\bf{a_i}}=\alpha_1^{a_{i,1}}\alpha_2^{a_{i,2}}\ldots\alpha_n^{a_{i,n}}$ $(i=1,2,\ldots,r)$.
\end{enumerate}

\begin{proof}
Since $\sigma|_{\bF}={\rm id}$, $\bF\subseteq\Const_{\sigma}\bF(\alpha_1,\alpha_2,\ldots,\alpha_n)$.
For each $1\leq i\leq r$, since ${\bf{a_i}}\in U_{\boldsymbol{\lambda}}$,
\[
\sigma({\boldsymbol{\alpha}}^{\bf{a_i}})
={\boldsymbol{\lambda}}^{\bf{a_i}}{\boldsymbol{\alpha}}^{\bf{a_i}}
={\boldsymbol{\alpha}}^{\bf{a_i}}.
\]
Then ${\boldsymbol{\alpha}}^{\bf{a_i}}\in\Const_{\sigma}\bF(\alpha_1,\alpha_2,\ldots,\alpha_n)$.
So $\bF({\boldsymbol{\alpha}}^{\bf{a_1}},{\boldsymbol{\alpha}}^{\bf{a_2}},\ldots,{\boldsymbol{\alpha}}^{\bf{a_r}})$
$\subseteq$ $\Const_{\sigma}\bF(\alpha_1,\alpha_2,\ldots,\alpha_n)$, since $\Const_{\sigma}\bF(\alpha_1,\alpha_2,\ldots,\alpha_n)$ is a field.

On the other hand, suppose $f\in \text{Const}_{\sigma}\bF(\alpha_1,\alpha_2,\ldots,\alpha_n)\setminus\{0\}$, and write $f=\frac{p}{q}$ with $p,q\in\bF[\alpha_1,\alpha_2,\ldots,\alpha_n]$ and $\gcd(p,q)=1$. Then
\begin{align}\label{eq-2-1}
\sigma\left(\frac{p}{q}\right)&=\frac{\sigma p}{\sigma q}=\frac{p}{q},
\end{align}
which implies $p\mid \sigma p$ and $q\mid\sigma q$ because $\gcd(p,q)=1$. Since $\sigma$ does not change the total degree of a polynomial, there exists $c\in\bF\setminus\{0\}$ such that $\sigma p=c\cdot p$ and $\sigma q=c\cdot q$,

Write
\[
p=c_{\bf0}\cdot{\boldsymbol{\alpha}}^{\bf{i_0}}\cdot(1+\sum_{{\bf i}\in\bN^n\setminus\{\bf0\}}c_{\bf{i}}\cdot{\boldsymbol{\alpha}}^{\bf{i}}),
\]
where $c_{\bf0}\in\bF\setminus\{0\}$, $c_{\bf i}\in\bF$, $\bf{i_0}\in\bN^n$. Then by $\sigma p=c\cdot p$, we have
\[
c_{\bf0}\cdot{\boldsymbol{\lambda}}^{\bf{i_0}}\cdot\boldsymbol{\alpha}^{\bf{i_0}}\cdot(1+\sum_{{\bf i}\in\bN^n\setminus\{\bf0\}}c_{\bf i}\cdot{\boldsymbol{\lambda}}^{\bf{i}}\cdot{\boldsymbol{\alpha}}^{\bf{i}})
=c\cdot c_{\bf0}\cdot{\boldsymbol{\alpha}}^{\bf{i_0}}\cdot(1+\sum_{{\bf i}\in\bN^n\setminus\{\bf0\}}c_{\bf i}\cdot{\boldsymbol{\alpha}}^{\bf{i}}).
\]
Comparing the coefficients of both sides yields that
\[\left\{
  \begin{array}{ll}
    {\boldsymbol{\lambda}}^{\bf{i_0}}=c \\
    c_{\bf i}\cdot\boldsymbol{\lambda}^{\bf{i_0}}\cdot \boldsymbol{\lambda}^{\bf{i}}=c\cdot c_{\bf i} \text{ for all } {\bf i}\in\bN^n\setminus\{\bf0\}
  \end{array}
\right.
.
\]
Since $c$ is nonzero, it implies that,
\[\left\{
  \begin{array}{ll}
    {\boldsymbol{\lambda}}^{\bf{i_0}}=c \\
    \text{either } c_{\bf i}=0 \text{ or }\boldsymbol{\lambda}^{\bf{i}}=1 \text{ for all } {\bf i}\in\bN^n\setminus\{\bf0\}
  \end{array}
\right.
.
\]

For $r=0$, we have $\boldsymbol{\lambda}^{\bf{i}}\neq1$ for all
${\bf i}\in\bN^n\setminus\{\bf0\}$, then $c_{\bf i}=0$ and $p=c_{\bf0}\cdot{\boldsymbol{\alpha}}^{\bf{i_0}}$.
Similarly, $q=b_{\bf0}\cdot{\boldsymbol{\alpha}}^{\bf{j_0}}$ for some $b_{\bf0}\in\bF\setminus\{0\}$, $\bf{j_0}\in\bN^n$.
By \eqref{eq-2-1}, we have
\[\frac{c_{\bf0}\cdot{\boldsymbol{\lambda}}^{\bf{i_0}}\cdot{\boldsymbol{\alpha}}^{\bf{i_0}}}{b_{\bf0}\cdot{\boldsymbol{\lambda}}^{\bf{j_0}}\cdot{\boldsymbol{\alpha}}^{\bf{j_0}}}=\frac{c_{\bf0}\cdot{\boldsymbol{\alpha}}^{\bf{i_0}}}{b_{\bf0}\cdot{\boldsymbol{\alpha}}^{\bf{j_0}}},\]
which implies ${\bf i_0=j_0=0}$, and $p/q\in\bF$.
So $\Const_{\sigma}\bF(\alpha_1,\alpha_2,\ldots,\alpha_n)\subseteq\bF$.

For $r\in\bN\setminus\{0\}$ and a fixed $\bf{i}\in\bN^n\setminus\{0\}$, if $c_{\bf i}=0$, then $0=c_{\bf{i}}\cdot{\boldsymbol{\alpha}}^{\bf{i}}\in\bF({\boldsymbol{\alpha}}^{\bf{a_1}},{\boldsymbol{\alpha}}^{\bf{a_2}},\ldots,{\boldsymbol{\alpha}}^{\bf{a_r}})$.
If $c_{\bf i}\neq0$, then $\boldsymbol{\lambda}^{\bf{i}}=1$.
So $\boldsymbol{\lambda}^{\bf{i}}\in U_{\boldsymbol{\lambda}}$ and we can write ${\bf{i}}=\sum_{j=1}^rs_j\bf{a_j}$ with $s_j\in\bZ$. Then
\[c_{\bf{i}}\cdot{\boldsymbol{\alpha}}^{\bf{i}}=c_{\bf{i}}\cdot{\boldsymbol{\alpha}}^{\sum_{j=1}^rs_j\bf{a_j}}
=c_{\bf{i}}\cdot\Pi_{j=1}^r(\boldsymbol{\alpha}^{\bf{a_j}})^{s_j}\in\bF({\boldsymbol{\alpha}}^{\bf{a_1}},{\boldsymbol{\alpha}}^{\bf{a_2}},\ldots,{\boldsymbol{\alpha}}^{\bf{a_r}}).\]
So $p=c_{\bf0}\cdot{\boldsymbol{\alpha}}^{\bf{i_0}}\cdot p_0$ with $p_0\in\bF({\boldsymbol{\alpha}}^{\bf{a_1}},{\boldsymbol{\alpha}}^{\bf{a_2}},\ldots,{\boldsymbol{\alpha}}^{\bf{a_r}})$. Similarly, $q=b_{\bf0}\cdot{\boldsymbol{\alpha}}^{\bf{j_0}}\cdot q_0$ for some $b_{\bf0}\in\bF\setminus\{0\}$, $\bf{j_0}\in\bN^n$ and  $q_0\in\bF({\boldsymbol{\alpha}}^{\bf{a_1}},{\boldsymbol{\alpha}}^{\bf{a_2}},\ldots,{\boldsymbol{\alpha}}^{\bf{a_r}})$.

Since $\bF({\boldsymbol{\alpha}}^{\bf{a_1}},{\boldsymbol{\alpha}}^{\bf{a_2}},\ldots,{\boldsymbol{\alpha}}^{\bf{a_r}})
\subseteq\text{Const}_{\sigma}\bF(\alpha_1,\alpha_2,\ldots,\alpha_n)$, we have $\sigma(p_0)=p_0$, $\sigma(q_0)=q_0$. By \eqref{eq-2-1},
\[\frac{c_{\bf0}\cdot{\boldsymbol{\lambda}}^{\bf{i_0}}\cdot{\boldsymbol{\alpha}}^{\bf{i_0}}\cdot p_0}{b_{\bf0}\cdot{\boldsymbol{\lambda}}^{\bf{j_0}}\cdot{\boldsymbol{\alpha}}^{\bf{j_0}}\cdot q_0}=\frac{c_{\bf0}\cdot{\boldsymbol{\alpha}}^{\bf{i_0}}\cdot p_0}{b_{\bf0}\cdot{\boldsymbol{\alpha}}^{\bf{j_0}}\cdot q_0},\]
which implies $\boldsymbol{\lambda}^{\bf{i_0-j_0}}=1$.
Using the same argument as above, we get
\[\boldsymbol{\alpha}^{\bf{i_0-j_0}}\in\bF({\boldsymbol{\alpha}}^{\bf{a_1}},{\boldsymbol{\alpha}}^{\bf{a_2}},\ldots,{\boldsymbol{\alpha}}^{\bf{a_r}}).
\]
So
\[
f=\frac{c_{\bf0}}{b_{\bf0}}\cdot\boldsymbol{\alpha}^{\bf{i_0-j_0}}\cdot\frac{p_0}{q_0}\in\bF({\boldsymbol{\alpha}}^{\bf{a_1}},{\boldsymbol{\alpha}}^{\bf{a_2}},\ldots,{\boldsymbol{\alpha}}^{\bf{a_r}}),
\]
which implies $\text{Const}_{\sigma}\bF(\alpha_1,\alpha_2,\ldots,\alpha_n)\subseteq\bF({\boldsymbol{\alpha}}^{\bf{a_1}},{\boldsymbol{\alpha}}^{\bf{a_2}},\ldots,{\boldsymbol{\alpha}}^{\bf{a_r}})$.

\end{proof}
\end{thm}
Note that we use the variant of Ge's algorithm given by Kauers et.al \citep{Manuel-2023} to compute the rank and a basic of $U_{\boldsymbol{\lambda}}$ in our package.

\subsection{The spread computation}
Let $(\bF(\alpha_1,\alpha_2,\ldots,\alpha_n),\sigma)$ be the multivariate difference field, which is defined by \eqref{eq-2-0}.
For any two polynomials $p,q\in\mathbb{F}[\alpha_1,\ldots,\alpha_n]$, let $\Spr_{\sigma}(p, q)$ denote the \emph{spread set} of $p$ and $q$ with respect to $\sigma$, i.e.,
\[\Spr_{\sigma}(p,q)=\{k\in\bZ| \sigma^{k}(p)=cq \text{ for some }c\in\bF\setminus\{0\}\}.\]
If $\Spr_{\sigma}(p, q)\neq\emptyset$, we say $p$ and $q$ are \emph{equivalent with respect to $\sigma$}, denoted by $p\sim_{\sigma}q$. Otherwise $p$ is not equivalent with $q$ and denote by $p\nsim_{\sigma}q$.

In order to compute the spread set $\Spr_{\sigma}(p, q)$ of polynomials $p$ and $q$, we show a result about $\Spr_{\sigma}(p, q)$  and $\Spr_{\sigma}(p, p)$.

\begin{lem}\label{lem-2-1}
Let $p,q\in\mathbb{F}[\alpha_1,\ldots,\alpha_n]$ be two polynomials. Then

$(1)$ $\Spr_{\sigma}(p,p)$ is subgroup of $\mathbb{Z}$.

$(2)$ If $k_0\in \Spr_{\sigma}(p,q)$, then $\Spr_{\sigma}(p, q)=k_0+\Spr_{\sigma}(p, p)$.

\begin{proof}
$(1)$ If $k_1,k_2\in \Spr_{\sigma}(p, p)$, then there exist $c_1,c_2\in\bF\setminus\{0\}$ such that
$\sigma^{k_1}p=c_1p$ and $\sigma^{k_2}p=c_2p$. So
\[p=\sigma^{-k_2}(c_2p)=c_2\sigma^{-k_2}(p)\]
and
\[\sigma^{k_1-k_2}(p)
=\sigma^{k_1}(\sigma^{-k_2}(p))
=\sigma^{k_1}(c_2^{-1}p)
=c_2^{-1}\sigma^{k_1}(p)
=c_2^{-1}c_1p.\]
This implies $k_1-k_2\in\Spr_{\sigma}(p, p)$. Thus $\Spr_{\sigma}(p, p)$ is a subgroup of $\bZ$.

$(2)$
Since $k_0\in \Spr_{\sigma}(p, q)$, there exists $c_0\in\bF\setminus\{0\}$ such that $\sigma^{k_0}(p)=c_0q$.

If $k\in\Spr_{\sigma}(p, q)$, then there exists $c\in\bF\setminus\{0\}$ such that $\sigma^{k}(p)=cq$. Moreover, we have
$\sigma^{k}(p)=cq=cc_0^{-1}(c_0q)=cc_0^{-1}\sigma^{k_0}(p)$.
So
\[\sigma^{k-k_0}(p)=cc_0^{-1}p\]
with $cc_0^{-1}\in\bF\setminus\{0\}$, which implies $k-k_0\in\Spr_{\sigma}(p, p)$.

If $h\in\Spr_{\sigma}(p, p)$, then there exists $d\in\bF\setminus\{0\}$ such that $\sigma^{h}(p)=dp$. Moreover, we have
\[\sigma^{k_0+h}(p)=\sigma^{k_0}(\sigma^h(p))=\sigma^{k_0}(dp)=d\sigma^{k_0}(p)
=dc_0p,\]
and
$dc_0\in\bF\setminus\{0\}$, which implies $k_0+h\in\Spr_{\sigma}(p, q)$.

Therefore $\Spr_{\sigma}(p, q)=k_0+\Spr_{\sigma}(p, p)$.
\end{proof}
\end{lem}

For a polynomial $p=\sum_{\bf i} p_{\bf i}\boldsymbol{\alpha}^{\bf i}\in\bF[\alpha_1,
\ldots,\alpha_n]$, where $p_{\bf i}\in\bF, {\bf i}=(i_1,i_2,\ldots,i_n)\in\bN^n$ and ${\boldsymbol{\alpha}}^{\bf{i}}=\alpha_1^{i_1}\alpha_2^{i_2}\ldots\alpha_n^{i_n}$, we denote the \emph{\text{support set}} of polynomial $p$ by $\supp(p)$, i.e.,
\[\supp(p)=\{\boldsymbol{\alpha}^{\bf i}|p_{\bf i}\neq0\}.\]

Let $p,q\in\bF[\alpha_1,
\ldots,\alpha_n]$ be two polynomials.
Without lose of generality, we assume $\supp(p)=\supp(q)$, otherwise $\Spr_{\sigma}(p, q)=\emptyset$.
We may further assume $p, q$ are monic and \begin{align}\label{eq-2-4}
p=\boldsymbol{\alpha}^{\bf m}+\sum_{\boldsymbol{\alpha}^{\bf i}<\boldsymbol{\alpha}^{\bf m}}p_{\bf i}\boldsymbol{\alpha}^{\bf i},\quad q=\boldsymbol{\alpha}^{\bf m}+\sum_{\boldsymbol{\alpha}^{\bf i}<\boldsymbol{\alpha}^{\bf m}}q_{\bf i}\boldsymbol{\alpha}^{\bf i},
\end{align}
where $p_{\bf i}, q_{\bf i}\in\bF$, ${\bf m}=(m_1,m_2,\ldots,m_n),{\bf i}=(i_1,i_2,\ldots,i_n)\in\bN^n$, ${\boldsymbol{\alpha}}^{\bf{m}}=\alpha_1^{m_1}\alpha_2^{m_2}\ldots\alpha_n^{m_n}$, and ${\boldsymbol{\alpha}}^{\bf i}=\alpha_1^{i_1}\alpha_2^{i_2}\ldots\alpha_n^{i_n}$.

\begin{lem}\label{lem-2-2}
Let $(\bF(\alpha_1,\alpha_2,\ldots,\alpha_n),\sigma)$ be the multivariate difference field, which is defined by \eqref{eq-2-0}. Let $p,q\in\mathbb{F}[\alpha_1,\ldots,\alpha_n]$ be two polynomials in \eqref{eq-2-4}.
Then
\[\Spr_{\sigma}(p,q)
=\bigcap_{\boldsymbol{\alpha}^{\bf i}<\boldsymbol{\alpha}^{\bf m}}
\Spr_{\sigma}(\widetilde{p_i}, \widetilde{q_i}),\]
where $\widetilde{p_{\bf i}}=\boldsymbol{\alpha}^{\bf m}+p_{\bf i}\boldsymbol{\alpha}^{\bf i}$ and $\widetilde{q_{\bf i}}=\boldsymbol{\alpha}^{\bf m}+q_{\bf i}\boldsymbol{\alpha}^{\bf i}$.
\begin{proof}
If $k\in\Spr_{\sigma}(p,q)$, then there exists $c\in\bF\setminus\{0\}$ such that $\sigma^k(p)=cq$, i.e.,
\[\sigma^k\left(\boldsymbol{\alpha}^{\bf m}+\sum_{\boldsymbol{\alpha}^{\bf i}<\boldsymbol{\alpha}^{\bf m}}p_{\bf i}\boldsymbol{\alpha}^{\bf i}\right)=c\cdot\left(\boldsymbol{\alpha}^{\bf m}+\sum_{\boldsymbol{\alpha}^{\bf i}<\boldsymbol{\alpha}^{\bf m}}q_{\bf i}\boldsymbol{\alpha}^{\bf i}\right).\]
So for any ${\bf i}$ with $\boldsymbol{\alpha}^{\bf i}<\boldsymbol{\alpha}^{\bf m}$, we have
\[\sigma^k\left(\boldsymbol{\alpha}^{\bf m}+p_{\bf i}\boldsymbol{\alpha}^{\bf i}\right)=c\cdot\left(\boldsymbol{\alpha}^{\bf m}+q_{\bf i}\boldsymbol{\alpha}^{\bf i}\right),\]
which implies
\[k\in\bigcap_{\boldsymbol{\alpha}^{\bf i}<\boldsymbol{\alpha}^{\bf m}}
\Spr_{\sigma}(\widetilde{p_i}, \widetilde{q_i}).\]

If $h\in\bigcap_{\boldsymbol{\alpha}^{\bf i}<\boldsymbol{\alpha}^{\bf m}}
\Spr_{\sigma}(\widetilde{p_i}, \widetilde{q_i})$, then, for any $\boldsymbol{\alpha}^{\bf i}$ with $\boldsymbol{\alpha}^{\bf i}<\boldsymbol{\alpha}^{\bf m}$, there exists $c_i\in\bF\setminus\{0\}$ such that
\[\sigma^h\left(\boldsymbol{\alpha}^{\bf m}+p_{\bf i}\boldsymbol{\alpha}^{\bf i}\right)=\boldsymbol{\lambda}^{\bf m}\boldsymbol{\alpha}^{\bf m}+p_{\bf i}\boldsymbol{\lambda}^{\bf i}\boldsymbol{\alpha}^{\bf i}=c_i\cdot\left(\boldsymbol{\alpha}^{\bf m}+q_{\bf i}\boldsymbol{\alpha}^{\bf i}\right),\]
where $\boldsymbol{\lambda}^{\bf m}=\lambda_1^{m_1}\lambda_2^{m_2}\cdots\lambda_n^{m_n}$. Comparing the coefficient of $\boldsymbol{\alpha}^{\bf m}$ yields that \[c_i=\boldsymbol{\lambda}^{\bf m}.\]
So
\[\sigma^h\left(\boldsymbol{\alpha}^{\bf m}+\sum_{\boldsymbol{\alpha}^{\bf i}<\boldsymbol{\alpha}^{\bf m}}p_{\bf i}\boldsymbol{\alpha}^{\bf i}\right)=\boldsymbol{\lambda}^{\bf m}\cdot\left(\boldsymbol{\alpha}^{\bf m}+\sum_{\boldsymbol{\alpha}^{\bf i}<\boldsymbol{\alpha}^{\bf m}}q_{\bf i}\boldsymbol{\alpha}^{\bf i}\right),\]
which implies
\[h\in
\Spr_{\sigma}(p,q).\]
\end{proof}
\end{lem}

\begin{prop}\label{pro-2-3}
Let $(\bF(\alpha_1,\alpha_2,\ldots,\alpha_n),\sigma)$ be the multivariate difference field, which is defined by \eqref{eq-2-0}. Let $p=\boldsymbol{\alpha}^{\bf m}+p_{\bf i}\boldsymbol{\alpha}^{\bf i}$, $q=\boldsymbol{\alpha}^{\bf m}+q_{\bf i}\boldsymbol{\alpha}^{\bf i}$ be two polynomials in $\mathbb{F}[\alpha_1,\ldots,\alpha_n]$ with $p_{\bf i}, q_{\bf i}\in\bF$. Then we can determine $\Spr_{
\sigma}(p,q)$.
\begin{proof}
Assume supp$(p)$=supp$(q)$, otherwise $\Spr_{
\sigma}(p,q)=\emptyset$. Also assume $q_{\bf i},p_{\bf i}\neq0$, otherwise $\Spr_{
\sigma}(p,q)=\Spr_{\sigma}(\boldsymbol{\alpha}^{\bf m},\boldsymbol{\alpha}^{\bf m})=\bZ$.

If there exists $k\in\Spr_{
\sigma}(p,q)$, then there exists $c\in\bF\setminus\{0\}$ such that $\sigma^k(p)=c q$. Note that $c$ is uniquely determined by $k$. We have
\[
\boldsymbol{\lambda}^{k\bf m}\cdot\boldsymbol{\alpha}^{\bf m}+p_{\bf i}\cdot\boldsymbol{\lambda}^{k{\bf i}}\cdot\boldsymbol{\alpha}^{\bf i}=c\left(\boldsymbol{\alpha}^{\bf m}+q_{\bf i}\boldsymbol{\alpha}^{\bf i}\right).
\]
Comparing the coefficients of different terms, we get
\begin{align}\label{eq-2-2}
\left\{
  \begin{array}{ll}
    \boldsymbol{\lambda}^{k\bf m}=c \\
   p_{\bf i}\cdot\boldsymbol{\lambda}^{k{\bf i}}=cq_{\bf i}
  \end{array}
\right.
,
\end{align}
which implies
\[p_{\bf i}/q_{\bf i}=\boldsymbol{\lambda}^{\bf I} \text{ for some } {\bf I}=(I_1, I_2,\ldots,I_n)\in\bZ^n,\]
and if $p_{\bf i}/q_{\bf i}\notin\{\boldsymbol{\lambda}^{\bf I}|{\bf I}=(I_1, I_2,\ldots,I_n)\in\bZ^n\}$, then $\Spr_{\sigma}(p,q)=\emptyset$.

Note that $(I_1,I_2,\ldots,I_n,-1)\in U(\lambda_1,\lambda_2,\ldots,\lambda_n, p_{\bf i}/q_{\bf i})$.
The indices $I_1,I_2,\ldots,I_n$ can be computed by the algorithm about the exponent lattice in \citep{Manuel-2023}.
Substituting the expression of $p_{\bf i}/q_{\bf i}$ into \eqref{eq-2-2} yields that
\[
\boldsymbol{\lambda}^{k\bf (m-i)-{\bf I}}=1.
\]
So
\[
\left(k\bf (m-i)-{\bf I}\right)\in U_{\boldsymbol{\lambda}}.
\]

Assume $U_{\boldsymbol{\lambda}}=\bigoplus_{i=1}^r\bZ{\bf{a_i}}$ for some ${\bf{a_i}}=(a_{i,1},a_{i,2},\ldots, a_{i,n})\in\bZ^n$. Note that $r\leq n$, then
\[\left(k\bf (m-i)-{\bf I}\right)
=\sum_{j=1}^r{\ell}_j\bf{a_j}
\]
for some $\ell_1,\cdots,\ell_r\in\bZ$.
This leads to a linear system with $r+1$ variables and $n$ equations
\begin{align}\label{eq-2-3}
\left(\begin{array}{cccc}
  m_1-i_1 & -a_{1,1} & \cdots & -a_{r,1} \\
m_2-i_2 & -a_{1,2} & \cdots & -a_{r,2} \\
  \vdots & \vdots & \ddots & \vdots \\
  m_n-i_n & -a_{1,n} & \cdots & -a_{r,n}
\end{array}\right)
\left( \begin{array}{c}
    k \\
    \ell_1 \\
    \vdots \\
    \ell_r \\
  \end{array}\right)
=
\left(\begin{array}{c}
    I_1\\
    I_2 \\
    \vdots \\
    I_n \\
  \end{array}\right).
\end{align}
So
\begin{align*}
\text{Spr}_{\sigma}(p,q)
&=\{k\in\bZ|\ \eqref{eq-2-3} \text{ holds  for some } \ell_1,\ldots,\ell_r\in\bZ\}.
\end{align*}
Note that if \eqref{eq-2-3} has no solution in $\bZ^{r+1}$, then $\Spr_{\sigma}(p,q)=\emptyset$.
Also if $p=q$, then $I_1= I_2=\ldots=I_n=0$. So
\begin{align*}
\Spr_{\sigma}(p,p)
&=\{k\in\bZ|\ \text{ associated homogeneous system of \eqref{eq-2-3} holds  for some } \ell_1,\ldots,\ell_r\in\bZ\}.
\end{align*}

According to Lemma \ref{lem-2-1}, $\Spr_{\sigma}(p, q)=k_0+\Spr_{\sigma}(p, p)$, where $(k_0, \ell_1, \ldots, \ell_r)\in\bZ^r$ is a solution of \eqref{eq-2-3} for some $\ell_1,\ldots,\ell_r$, and $\Spr_{\sigma}(p, p)$ is a subgroup of $\bZ$.
In order to character the spread set $\Spr_{\sigma}(p, q)$, next it suffices to find $\ell_0\in\bZ$ such that $\ell_0\bZ=\Spr_{\sigma}(p,p)$.

Let $s\in\bN$ denote the rank of the solution space of associated homogeneous system of \eqref{eq-2-3}. When $s=0$, the homogeneous system has the unique solution ${\bf 0}$, then $\text{Spr}_{\sigma}(p,p)=\{0\}$, and $\text{Spr}_{\sigma}(p, q)=k_0$.
When $s\geq1$, let ${{\boldsymbol \ell}_i}=({\ell}_{i0}, {\ell}_{i1},\ldots, {\ell}_{ir})$ be a basis for the solution space for $s\geq 1$ and $i=1,2,\ldots,s$ . Then
\begin{align}\label{eq-2-6}
\text{Spr}_{\sigma}(p,p)
=&\{{\ell_0}\in\bZ|({\ell}_0,{\ell}_1,\ldots,{\ell}_r)\in\sum_{j=1}^s\bZ{\bf {\ell}_i} \text{ for some } {\ell}_1,\ldots,{\ell}_r\in \bZ\}\nonumber\\
=&\{{\ell}_0\in\bZ|({\ell}_0,{\ell}_1,\ldots,{\ell}_r)=\left(\sum_{j=1}^sy_j{\ell}_{j,0},\ldots,\sum_{j=1}^sy_j{\ell}_{j,r}\right) \text{ for some } {\ell}_1,\ldots,{\ell}_r,y_1\ldots,y_r\in \bZ\}\nonumber\\
=&\{{\ell}_0\in\bZ|{\ell}_0=\sum_{j=1}^sy_j{\ell}_{j,0}\text{ for some } y_1\ldots,y_r\in \bZ\}\nonumber\\
=&\gcd({\ell}_{1,0},{\ell}_{2,0},\ldots,{\ell}_{s,0})\bZ,
\end{align}
i.e., for $\ell_0=\gcd(\ell_{1,0},\ell_{2,0},\ldots,\ell_{s,0})$, we have
\[\Spr_{\sigma}(p,p)=\ell_0\bZ.\]
So
\[\Spr_{\sigma}(p, q)=k_0+\ell_0\bZ.\]
\end{proof}
\end{prop}

Base on Lemma \ref{lem-2-2} and Proposition \ref{pro-2-3}, we give an algorithm for computing $\text{Spr}_{\sigma}(p,q)$ as follows.

\begin{breakablealgorithm}\label{alg:spread}
\caption{An algorithm for computing the spread set of two polynomials}
\begin{algorithmic}[htb]
	\textbf{Input:}  two polynomials $p,q\in\bF[\alpha_1,
\ldots,\alpha_n]$\\
	\textbf{Output:} the spread set $\Spr_{\sigma}(p,q)$
	
    \step 11 if $\supp(p)\neq\supp(q)$, then {\em\bf return} $\Spr_{\sigma}(p,q)=\emptyset$;
    \step 21 if $\supp(p)=\supp(q)$ and $|\supp(p)|=1$, then {\em\bf return} $\Spr_{\sigma}(p,q)=\bZ$;
    \step 31 if $\supp(p)=\supp(q)$ and $|\supp(p)|>1$,
    \step { }2 replace $p$ by $p/\lc(p)$ and $q$ by $q/\lc(q)$, and write $p,q$ as the form in \eqref{eq-2-4},\\
    \step{ }2 let $T=\supp(p)\backslash\{\lt(p)\}$, $k_0=0$, $\ell_0=1$, $c_0=1$, $c_1= \boldsymbol{\lambda}^{\bf m}$.
    \step 42 while $T\neq\emptyset$, do
    \step 53 pick ${\boldsymbol{\alpha}}^{\bf i}={\alpha_1}^{i_1}{\alpha_2}^{i_2}\cdots{\alpha_n}^{i_n}\in T$, and replace $T$ by $T\backslash\{{\boldsymbol{\alpha}}^{\bf i}\}$.
    \step 63 test whether
\begin{align}
\sigma^{k_0}(lt(p)+p_{{\bf i}}{\boldsymbol{\alpha}}^{\bf i})=c_0(lt(q)+q_{{\bf i}}{\boldsymbol{\alpha}}^{\bf i})\label{eq-2-7},\\
\sigma^{\ell_0}(lt(p)+p_{\bf i}{\boldsymbol{\alpha}}^{\bf i})=c_1(lt(p)+p_{\bf i}{\boldsymbol{\alpha}}^{\bf i})\label{eq-2-8}.
\end{align}
    \step 74 if $k_0$ is not a solution of \eqref{eq-2-7} or $\ell_0$ is not a solution of \eqref{eq-2-8}, then
    \step { }5 use Proposition \ref{pro-2-3} to find all $k\in\bZ$ such that
\[\sigma^{k}(lt(p)+p_{{\bf i}}{\boldsymbol{\alpha}}^{\bf i})=c(lt(q)+q_{{\bf i}}{\boldsymbol{\alpha}}^{\bf i})\]
    \step { }5 for some $c\in\bF$.
    \step {8}6 if there is no such $k$, then {\em\bf return} $\emptyset$;
    \step {9}6 else
    \step { }7 write its solution as $k_0'+\ell_0'\bZ$.
    \step {10}7 if $k_0+\ell_0\bZ\cap k_0'+\ell_0'\bZ=\emptyset$, then {\em\bf return} $\emptyset$;
    \step {11}7 else
    \step {12}8 replace $k_0+\ell_0\bZ$ by $(k_0+\ell_0\bZ)\cap(k_0'+\ell_0'\bZ)$;
    \step {  }8 replace $c_0$ by $\boldsymbol{\lambda}^{ k_0{\bf m}}$;
    replace $c_1$ by $\boldsymbol{\lambda}^{\ell_0 {\bf m}}$.
    \step {13}2 {\em\bf return} $\Spr_{\sigma}(p,q)=k_0+\ell_0\bZ$.
\end{algorithmic}
\end{breakablealgorithm}
Algorithm \ref{alg:spread} can be early terminated if $\Spr_{\sigma}(p,q)=\emptyset$. An implementation of the algorithm is a part of the package \texttt{difference\_field}.

Next we show some examples about the spread computation in the multivariate difference field $(\bQ(\alpha_1,\alpha_2,\alpha_3),\sigma)$ with
\[\sigma|_{\bQ}={\rm id},\ {\rm and}\  (\sigma(\alpha_1,\alpha_2,\alpha_3))=(\alpha_1,\alpha_2,\alpha_3)\diag(-1,1/2,-4).
\]
Let $\boldsymbol{\lambda}=(-1,1/2,-4)$, then the exponent lattice $U_{\boldsymbol{\lambda}}$ $=(1,2,1)\bZ\oplus(0,4,2)\bZ$, and $\Const_{\sigma}\bQ(\alpha_1,\alpha_2,\alpha_3)$ $=$ $\bQ(\alpha_1{\alpha_2}^2\alpha_3,{\alpha_2}^4{\alpha_3}^2)$.

\begin{exam}
(1) Let $p=\alpha_1+\alpha_3$, $q=\frac{1}{4}\alpha_1+\alpha_3$. Then $\Spr_{\sigma}(p,q)=\{1\}$ is a single point.
Indeed $\sigma(p)=-\alpha_1-4\alpha_3=-4q$.

(2) Let $p=\alpha_1+\alpha_3$, $q=2{\alpha_2\alpha_3}-1$. Then $\Spr_{\sigma}(p,q)=\emptyset$, since $\supp(p)\neq\supp(q)$.

(3) Let $p=2{\alpha_1}^2\alpha_3-1$. Then $\Spr_{\sigma}(p,p)=2\bZ$. Indeed $\sigma^2(p)=p$, $\sigma(p)=-2{\alpha_1}^2\alpha_3-1\neq 2p$ for all $c\in\bF\setminus\{0\}$.
\end{exam}

\begin{exam}
Let $p={\alpha_2}^2\alpha_3+\alpha_2\alpha_3+1$, $q={\alpha_2}^2\alpha_3+8\alpha_2\alpha_3-1$. Compute the spread set $\Spr_{\sigma}(p,q)$ of $p$ and $q$.
\end{exam}

\begin{proof}
By the definition of $p$ and $q$. we have $\supp(p)=\supp(q)=\{\alpha_2^2\alpha_3, \alpha_2\alpha_3,1\}$ and $\alpha_2^2\alpha_3\succ\alpha_2\alpha_3\succ1$.

Let $T=\supp(p)\backslash\{\lt(p)\}$, $k_0=0$, $l_0=1$, $c_0=1$, $c_1=\lambda_2^{2}\alpha_3$.

Pick $1\in T$, and replace $T$ by $T\backslash\{1\}$.
Let $p_1={\alpha_2}^2\alpha_3+1$, $q_1={\alpha_2}^2\alpha_3-1$.
Since $\sigma^{k_0}(p_1)\neq c_0(q_1)$,
we want to compute all $k\in\bZ$ such that
\[\sigma^{k}(p_1)=cq_1,\]
for some $c\in\bF$. According to Proposition \ref{pro-2-3}, we need to consider the linear system \eqref{eq-2-3}.

Let $m_1=0, m_2=2, m_3=1$, $i_1=i_2=i_3=0$. And, let $I_1=-1,I_2=0,I_3=0$, since $1/(-1)=(-1)^{-1}(1/2)^0(-4)^0$.
Moreover, by $U=(1,2,1)\bZ\oplus(0,4,2)\bZ$, we have $a_{1,1}=1, a_{1,2}=2, a_{1,3}=1$, and $a_{2,1}=0, a_{2,2}=4, a_{2,3}=2$.
Substituting into the linear system \eqref{eq-2-3}, we have
\begin{align*}
\left(\begin{array}{cccc}
  0 & -1 & 0 \\
2 & -2 & -4 \\
  1& -1 & -2
\end{array}\right)
\left( \begin{array}{c}
    k \\
    \ell_1 \\
    \ell_2 \\
  \end{array}\right)
=
\left(\begin{array}{c}
   -1\\
    0 \\
    0 \\
  \end{array}\right).
\end{align*}
Its solution set in $\bZ^3$ is $(1,1,0)+(2,0,1)\bZ$, so $\Spr_{\sigma}(p_1, q_1)=1+2\bZ$.

Update $k_0=1$, $l_0=2$.
Then $\sigma(p_1)=c_0q_1$, $\sigma^{2}(p_1)=c_1(p_1)$, where $c_0=((1/2)^{2}(-4))^{k_0}=-1$, $c_1=((1/2)^{2}(-4))^{l_0}=1$.

Pick $\{\alpha_2\alpha_3\}\in T$, and replace $T$ by $T\backslash\{\alpha_2\alpha_3\}$. Let $p_2={\alpha_2}^2\alpha_3+{\alpha_2}\alpha_3$, $q_1={\alpha_2}^2\alpha_3+8{\alpha_2}\alpha_3$.
Since $\sigma^{k_0}(p_2)\neq c_0(q_2)$,
we want to compute all $k\in\bZ$ such that
\[\sigma^{k}(p_2)=cq_2.\]
for some $c\in\bF$. According to Proposition \ref{pro-2-3}, we need to consider the linear system \eqref{eq-2-3}.

Let $m_1=0, m_2=2, m_3=1$, $i_1=0, i_2=1, i_3=1$. And, let $I_1=0,I_2=-1,I_3=-2$, since $1/8=(-1)^{0}(1/2)^{-1}(-4)^{-2}$.
Moreover, by $U=(1,2,1)\bZ\oplus(0,4,2)\bZ$, we have $a_{1,1}=1, a_{1,2}=2, a_{1,3}=1$, and $a_{2,1}=0, a_{2,2}=4, a_{2,3}=2$.
Substituting into the linear system \eqref{eq-2-3}, we have
\begin{align*}
\left(\begin{array}{cccc}
  0 & -1 & 0 \\
1 & -2 & -4 \\
  0& -1 & -2
\end{array}\right)
\left( \begin{array}{c}
    k \\
    \ell_1 \\
    \ell_2 \\
  \end{array}\right)
=
\left(\begin{array}{c}
   0\\
    -1 \\
    -2 \\
  \end{array}\right).
\end{align*}
It has a unique solution in $\bZ^3$ is $(3,0,1)$, so $\Spr_{\sigma}(p_2, q_2)=\{3\}$.

We have $\Spr_{\sigma}(p_1, q_1)\cap\Spr_{\sigma}(p_2, q_2)=(1+2\bZ)\cap\{3\}={3}$.
Update $k_0=3$, $l_0=0$.
Then $\sigma^3(p_2)=c_0q_2$, $\sigma^{0}(p_2)=c_1(p_2)$, where $c_0=((1/2)^{2}(-4))^{k_0}=-1$, $c_1=((1/2)^{2}(-4))^{l_0}=1$.

Return $\Spr_{\sigma}(p,q)=\{3\}$, since $T=\emptyset$.
\end{proof}

\section{The existence of rational solutions}\label{sec-3}
In this section, let $(\bF(\alpha_1,\alpha_2,\ldots,\alpha_n),\sigma)$ be the multivariate difference field, which is defined by \eqref{eq-2-0}, $c\in\bF\setminus\{0\}$ and $f\in\bF(\alpha_1,\alpha_2,\ldots,\alpha_n)$, we consider the $c\sigma$-summability problem of $f$ in $(\bF(\alpha_1,\alpha_2,\ldots,\alpha_{n})$, $\sigma)$, i.e, whether there exists $g\in\bF(\alpha_1, \alpha_2,\ldots,\alpha_n)$ such that
 \begin{align}\label{eq-3-3}
 f=c\sigma(g)-g,
 \end{align}
and how to find one $g$ if such a $g$ exists. We call $f$ is \emph{$c\sigma$-summable} in multivariate difference field $(\bF(\alpha_1,\alpha_2,\ldots,\alpha_{n}),\sigma)$ and denoted by $f=\Delta_{c\sigma}(g)$ when such equation holds.

Let $G=\langle \sigma\rangle=\langle\sigma^i|i\in\bZ\rangle$, for given polynomials $p,q\in\bF[\alpha_1\alpha_2,\ldots,\alpha_n]$,
the set $[p]_G=\{\tau(p)|\tau\in G\}$ is called the \emph{$G$-orbit} of $p$.
Recall that we say $p$ is equivalent with $q$ with respect to $\sigma$, denoted by $p\sim_{\sigma}q$, if $\Spr_{\sigma}(p, q)\neq\emptyset$. Then $[p]_G=[q]_G$ if and only if $p\sim_{\sigma}q$.

We mainly show the existence of rational solutions of the equation \eqref{eq-3-3} by the following three subsections.

%
%

\subsection{Decomposition of rational functions}
Let $f=P/Q\in\bF(\alpha_1,\alpha_2,\ldots,\alpha_{n})$ be a rational function, where $P,Q\in\bF(\alpha_1,\alpha_2,\ldots,\alpha_{n-1})[\alpha_n]$ are relatively prime. Then we can do a factorization
\[Q=c_q\cdot\prod_{i=1}^Iq_i^{k_i},\]
where $c_q\in\bF(\alpha_1,\alpha_2,\ldots,\alpha_{n-1}), I, k_i\in\bN\setminus\{0\}$, and $q_i\in\bF[\alpha_1,\alpha_2,\ldots,\alpha_n]$ are monic pairwise coprime irreducible polynomials with positive degree in $\alpha_n$.
By Algorithm \ref{alg:spread}, we can check whether $q_i\nsim_{\sigma} q_{i'}$ for any $i\neq i'$, and classify all of the polynomials $q_i$ into distinct $G$-orbit,
which leads to a fixed factorization
\[Q=c_q'\cdot\prod_{i=1}^I\prod_{j=1}^{J_i}\sigma^{\ell_{i,j}}(d_i)^{k_{i,j}},\]
where $c_q'\in\bF(\alpha_1,\alpha_2,\ldots,\alpha_{n-1}), I, J_i, k_{i,j}\in\bN\setminus\{0\}, \ell_{i,j}\in\bZ, d_i\in\bF[\alpha_1,\alpha_2,\ldots,\alpha_n]$ are monic irreducible polynomials with positive degree in $\alpha_n$. Moreover, $d_i$ are in distinct $G$-orbits, and for all $1\leq \ell_{i,j}\neq \ell_{i,j'}\leq J_i$, $\sigma^{\ell_{i,j}}(d_i)\neq \sigma^{\ell_{i,j'}}(d_i)$.
Based on this fixed factorization, for any rational function $f=P/Q\in\bF(\alpha_1,\alpha_2,\ldots,\alpha_n)$, it could be written as the irreducible partial faction decomposition of the form
\begin{align}\label{eq-3-1}
f=p+\sum_{j=1}^J\sum_{i=1}^{I_j}\sum_{l=1}^{e_{i,j}}\frac{a_{i,j,l}}{\sigma^{\ell}(d_i)^j},
\end{align}
where $p\in\bF(\alpha_1,\alpha_2,\ldots,\alpha_{n-1})[\alpha_n,\alpha_n^{-1}]$,
$J,I_j,e_{i,j}\in\bN\setminus\{0\}$, and $a_{i,j,l}\in\bF(\alpha_1,\alpha_2,\ldots,\alpha_{n-1})[\alpha_n]$ with $\deg_{\alpha_n}(a_{i,j,l})<\deg_{\alpha_n}(d_i)$.

Note that the decomposition \eqref{eq-3-1} depends on the choice of representatives $d_i$ in distinct $G$-orbits.
In order to obtain a unique decomposition of a rational function, we define one vector space, for given an irreducible polynomial $d\in\bF[\alpha_1,\alpha_2,\ldots,\alpha_{n}]$ and $j\in\bN\setminus\{0\}$.
\begin{align*}
V_{[d]_{G,j}}=\text{span}_{\bF(\alpha_1\alpha_2,\ldots,\alpha_{n-1})}
\left\{\frac{a}{\theta(d)^j}\bigg|
a\in\bF(\alpha_1,\alpha_2,\ldots,\alpha_{n-1})[\alpha_n],
\theta\in G,d\nsim_{\sigma}\alpha_n,\deg_{\alpha_n}(a)<\deg_{\alpha_n}(d) \right\}.
\end{align*}
Also, we definite vector spaces to decompose $p$. For each $i\in\bZ$,
\begin{align*}
V_{i}=\text{span}_{\bF(\alpha_1,\alpha_2,\ldots,\alpha_{n-1})}\left\{\alpha_n^i
\right\}.
\end{align*}
Then any rational function $f$ can be uniquely written in the form
\begin{align}\label{eq-3-2}
f=\sum_{i\in\bZ}p_i\alpha_n^i+\sum_{j=1}^J\sum_{[d]_{G}}f_{[d]_{G,j}},
\end{align}
where $J\in\bN\setminus\{0\}$, $p_i\in\bF(\alpha_1,\alpha_2,\ldots,\alpha_{n-1})$, $p_i\alpha_n^i$ are in distinct $V_i$ space and $f_{[d]_{G,j}}$ are in distinct $V_{[d]_{G,j}}$ space. The decomposition \eqref{eq-3-2} is called an \emph{orbital partial fraction decomposition} of $f\in\bF(\alpha_1,\alpha_2,\ldots,\alpha_n)$ with respect to $\alpha_n$ and $\sigma$.

\subsection{Orbital reduction for summability}
First, we transform the summabilty problem in the multivariate difference field $\bF(\alpha_1,\alpha_2,\ldots,\alpha_{n})$ into the summability problem in vector spaces $V_i$ and $V_{[d]_{G,j}}$.
Let $\Omega=\bZ\cup\{[d]_{G,j}\ |\ d$ is a irreducible polynomial in $ \bF[\alpha_1,\alpha_2,\ldots,\alpha_{n}]$ with $\deg_{\alpha_n}(d)\geq1$ and $d\nsim_{\sigma}\alpha_n,j\in\bN^{+}\}$.

\begin{lem}\label{lem-3-1}
If $f\in V_{\omega}$ for some $\omega\in\Omega$ and $L\in\bF(\alpha_1,\alpha_2,\ldots,\alpha_{n-1})[G]$, then $L(f)\in V_{\omega}$.
\end{lem}
\begin{proof}
Write $L=\sum_{\theta}p_{\theta}\cdot\theta$ with $p_{\theta}\in\bF(\alpha_1,\alpha_2,\ldots,\alpha_{n-1})$ and $\theta\in G$.

If $f\in V_i$ for $i\in\bZ$, then write $f=f_i\alpha_n^i$ with $f_i\in\bF(\alpha_1,\alpha_2,\ldots,\alpha_{n-1})$. We have
\[
p_{\theta}\cdot\theta(f)=p_{\theta}\cdot\theta(f_i\alpha_n^i)=p_{\theta}\cdot\theta(f_i)\cdot\theta(\alpha_n)^i.
\]
Moreover $\theta(f_i)\in\bF(\alpha_1,\alpha_2,\ldots,\alpha_{n-1})$ and $\theta(\alpha_n)^i\in V_i$, then
$p_{\theta}\cdot\theta(f)\in V_i$.
So $L(f)\in V_i$, since $V_i$ is an $\bF(\alpha_1,\alpha_2,\ldots,\alpha_{n-1})$-vector space.

If $f\in V_{[d]_{G,j}}$, then write $f=\sum_{i}\frac{a_i}{b_i^j}$ with $a_i\in\bF(\alpha_1,\alpha_2,\ldots,\alpha_{n-1})[\alpha_n]$, $b_i\in\bF[\alpha_1,\alpha_2,\ldots,\alpha_n]$, $b_i\sim_{\sigma}d$ and $\deg_{\alpha_n}(a_i)<\deg_{\alpha_n}(d)$. We have
\[
p_{\theta}\cdot\theta(f)=\sum_{i}\frac{p_{\theta}\cdot\theta(a_i)}{\theta(b_i)^j}.\]
Moreover $\theta(b_i)\sim_{\sigma}d$ and $\deg_{\alpha_n}(p_{\theta}\cdot\theta(a_i))<\deg_{\alpha_n}(\theta(b_i)^j)$.
So $p_{\theta}\cdot\theta(f)\in V_{[d]_{G,j}}$ which implies $L(f)\in V_{[d]_{G,j}}$.
\end{proof}

\begin{lem}\label{lem-3-2}
Let $f\in\bF(\alpha_1,\alpha_2,\ldots,\alpha_{n})$ and write $f=\sum_{\omega\in\Omega}f_{\omega}$.
Then $f$ is $c\sigma$-summable in $\bF(\alpha_1,\alpha_2,\ldots,\alpha_{n})$  if and only if each $f_{\omega}$ is $c\sigma$-summable in $V_{\omega}$ for all $\omega\in\Omega$.
\end{lem}
\begin{proof}
Assume $f_{\omega}=c\sigma(g_{\omega})-g_{\omega}$ for some $g_{\omega}\in V_{\omega}$. Then
\[f=\sum_{\omega\in\Omega}f_{\omega}=\sum_{\omega\in\Omega}\left(c\sigma(g_{\omega})-g_{\omega}\right)=
c\sigma\left(\sum_{\omega\in\Omega}g_{\omega}\right)-\sum_{\omega\in\Omega}g_{\omega},\]
which implies $f$ is $c\sigma$-summable.

Conversely, assume $f=c\sigma(g)-g$. Write $g=\sum_{\omega\in\Omega}g_{\omega}$ with $g_{\omega}\in V_{\omega}$. Then
\[f=c\sigma\left(\sum_{\omega\in\Omega}g_{\omega}\right)-\sum_{\omega\in\Omega}g_{\omega}=\sum_{\omega\in\Omega}\left(c\sigma(g_{\omega})-g_{\omega}\right).\]
By Lemma \ref{lem-3-1}, we have $\left(c\sigma(g_{\omega})-g_{\omega}\right)\in V_{\omega}$.
Since the orbital partial decomposition of $f$ is unique, $f_{\omega}=c\sigma(g_{\omega})-g_{\omega}$, which implies $f_{\omega}$ is $c\sigma$-summable in $V_{\omega}$.
\end{proof}

Then, based on the orbital partial fraction decomposition of $f$ with respect to $\sigma$,
it suffices to consider the $c\sigma$-summability problems of $p_i\alpha_n^i$ in $V_i$ and $f_{[d]_{G,j}}$ in $V_{[d]_{G,j}}$, respectively.

\subsubsection{Reduction of the $c\sigma$-summability problem of rational functions in $V_{i}$}
We only need consider the $c\sigma$-summability of $f_i(\alpha_1,\ldots,\alpha_{n-1})\alpha_n^i$ in $V_i$ for each fixed $i\in\bZ$.

\begin{lem}\label{lem-3-3}
Let $f\in\bF(\alpha_1,\ldots,\alpha_{n-1})$ and $i\in\bZ$. Then
$f(\alpha_1,\ldots,\alpha_{n-1})\alpha_n^i$ is $c\sigma$-summable in $V_i$ if and only if $f(\alpha_1,\ldots,\alpha_{n-1})$ is $c'\sigma$-summable in $\bF(\alpha_1,\ldots,\alpha_{n-1})$ with $c'=c\lambda_n^i$.
\end{lem}
\begin{proof}
It is trivial for the sufficiency.  
For the necessity, assume $f(\alpha_1,\ldots,\alpha_{n-1})\alpha_n^i$ is $c\sigma$-summable in $V_i$. Then $f(\alpha_1,\ldots,\alpha_{n-1})\alpha_n^i=c\sigma(g(\alpha_1,\ldots,\alpha_{n-1})\alpha_n^i)-g(\alpha_1,\ldots,\alpha_{n-1})\alpha_n^i$ for some $g(\alpha_1,\ldots,\alpha_{n-1})$ $\in$ $\bF(\alpha_1,\ldots,\alpha_{n-1})$,
and
\begin{align*}
f(\alpha_1,\ldots,\alpha_{n-1})\alpha_n^i&=c\lambda_n^i\sigma(g(\alpha_1,\ldots,\alpha_{n-1}))\alpha_n^i-g(\alpha_1,\ldots,\alpha_{n-1})\alpha_n^i.
\end{align*}
Since both $f$ and $g$ are free of $\alpha_n$, it follows that
\begin{align*}
f(\alpha_1,\ldots,\alpha_{n-1})&=c\lambda_n^i\sigma(g(\alpha_1,\ldots,\alpha_{n-1}))-g(\alpha_1,\ldots,\alpha_{n-1}),
\end{align*}
which implies $f(\alpha_1,\ldots,\alpha_{n-1})$ is $c'\sigma$-summable in $\bF(\alpha_1,\ldots,\alpha_{n-1})$ with $c'=c\lambda_n^i$.
\end{proof}

Note that we reduce the summablity problem in $V_i$ with $n$ variables to the summability problem in $V_i$ with $n-1$ variables. Now we consider the summability problem in $V_i$ with $n=1$.

\begin{thm}
For $f\neq0\in\bF$, then the following conditions are equivalent:

(1) $f\alpha_1^i$ is $c\sigma$-summable in $V_i$;

(2) $f$ is $c'\sigma$-summable in $\bF$ with $c'=c\lambda_1^i$;

(3) $\lambda_1^i\neq1/c$.
\end{thm}
\begin{proof}
 It is obviously (1)$\Leftrightarrow$(2) by the Lemma \ref{lem-3-3}. For the (2)$\Leftrightarrow(3)$, we have
\begin{align*}
f \text{ is } c'\sigma-\text{summable in }\bF\Leftrightarrow f&=c\lambda_i^i\sigma (u)-u\\
&=u(c\lambda_i^i-1)\text{ for some }u\in\bF\\
\Leftrightarrow c&\lambda_i^i-1\neq0\\
\Leftrightarrow\lambda&_i^i\neq1/c.
\end{align*}
\end{proof}


\subsubsection{Reduction of the $c\sigma$-summability problem of rational functions in $V_{[d]_{G,j}}$}

\begin{lem}\label{lem-3-4}
For $f\in V_{[d]_{G,j}}$, write $f=\sum_{\ell\in \bZ}\frac{a_{\ell}}{{\sigma^{\ell}}(d)^j}$, where $a_{\ell}\in\bF(\alpha_1,\alpha_2,\ldots,\alpha_{n-1})[\alpha_n]$, $\deg_{\alpha_n}(a_{\ell})$ $<$ $\deg_{\alpha_n}(d)$. Then $f$ can be decomposed into the form
\[f=\Delta_{c\sigma}(g)+r, \text{ with } r=\frac{a}{d^j},\]
where
$g\in V_{[d]_{G,j}}, a=\sum_{\ell}(c\sigma)^{-\ell}(a_{\ell})$ with $\deg_{\alpha_n}(a)<\deg_{\alpha_n}(d)$.
In particular, $f$ is $c\sigma$-summable in  $V_{[d]_{G,j}}$
if and only if
$r$ is $c\sigma$-summable in $V_{[d]_{G,j}}$.
\end{lem}
\begin{proof}
It suffices to prove that for each $\ell\in\bZ$, $\frac{a}{\sigma^{\ell}(d)^j}=\Delta_{c\sigma}(g_{\ell})+\frac{(c\sigma)^{-{\ell}}(a)}{d^j}$ for some $g_{\ell}\in V_{[d]_{G,j}}$.

If $\ell=0$, take $g_{\ell}=0$. If ${\ell}l>0$, then
\begin{align*}
\frac{a}{\sigma^{\ell}(d)^j}&=\frac{a}{\sigma^{\ell}(d)^j}-\frac{c^{-1}\sigma^{-1}(a)}{\sigma^{{\ell}-1}(d)^j}+\frac{c^{-1}\sigma^{-1}(a)}{\sigma^{{\ell}-1}(d)^j}\\
&=c\sigma\left(\frac{c^{-1}\sigma^{-1}(a)}{\sigma^{{\ell}-1}(d)^j}\right)-\frac{c^{-1}\sigma^{-1}(a)}{\sigma^{{\ell}-1}(d)^j}+\frac{c^{-1}\sigma^{-1}(a)}{\sigma^{{\ell}-1}(d)^j}\\
&=\Delta_{c\sigma}\left(\frac{c^{-1}\sigma^{-1}(a)}{\sigma^{{\ell}-1}(d)^j}\right)+\frac{c^{-1}\sigma^{-1}(a)}{\sigma^{{\ell}-1}(d)^j}\\
&=\Delta_{c\sigma}\left(\frac{c^{-1}\sigma^{-1}(a)}{\sigma^{{\ell}-1}(d)^j}+\frac{c^{-2}\sigma^{-2}(a)}{\sigma^{{\ell}-2}(d)^j}+\cdots+\frac{c^{-{\ell}}\sigma^{-{\ell}}(a)}{d^j}\right)+\frac{(c\sigma)^{-{\ell}}(a)}{d^j}.
\end{align*}
And for ${\ell}<0$,
\begin{align*}
\frac{a}{\sigma^{\ell}(d)^j}&=\frac{a}{\sigma^{\ell}(d)^j}-\frac{c\sigma(a)}{\sigma^{{\ell}+1}(d)^j}+\frac{c\sigma(a)}{\sigma^{{\ell}+1}(d)^j}\\
&=\frac{a}{\sigma^{\ell}(d)^j}-c\sigma\left(\frac{a}{\sigma^{\ell}(d)^j}\right)+\frac{c\sigma(a)}{\sigma^{{\ell}+1}(d)^j}\\
&=\Delta_{c\sigma}\left(-\frac{a}{\sigma^{\ell}(d)^j}\right)+\frac{c\sigma(a)}{\sigma^{{\ell}+1}(d)^j}\\
&=\Delta_{c\sigma}\left(-\frac{a}{\sigma^{\ell}(d)^j}-\frac{c\sigma(a)}{\sigma^{{\ell}+2}(d)^j}-\cdots-\frac{(c\sigma)^{-{\ell}-1}(a)}{\sigma(d)^j}\right)+\frac{(c\sigma)^{-{\ell}}(a)}{d^j}.
\end{align*}
In summary,
\[\frac{a}{\sigma^{\ell}(d)^j}=\Delta_{c\sigma}(g_{\ell})+\frac{(c\sigma)^{-{\ell}}(a)}{d^j},\]
where
\[
g_l=\left\{
       \begin{array}{ll}
         \sum_{i=0}^{{\ell}-1}\frac{(c\sigma)^{i-{\ell}}(a)}{\sigma^i(d)^j}, & \hbox{if ${\ell}>0$;} \\
         0, & \hbox{if ${\ell}=0$;} \\
         -\sum_{i=0}^{-{\ell}-1}\frac{(c\sigma)^{i}(a)}{\sigma^{{\ell}+i}(d)^j}, & \hbox{if ${\ell}<0$.}
       \end{array}
     \right.
\]
\end{proof}
Then it suffices to consider the $c\sigma$-summability criteria for $\frac{a}{d^j}$ in $V_{[d]_{G,j}}$.

\subsection{$c\sigma$-summability criterion for simple fractions}

For given an irreducible polynomial $d\in\bF[\alpha_1,\alpha_2,\ldots,\alpha_{n}]$ with $d\nsim_{\sigma}\alpha_n$, we call
$\frac{a}{d^j}$ is a \emph{simple fraction} in $\bF(\alpha_1,\alpha_2,\ldots,\alpha_{n})$ , where
 $j\in\bN\setminus\{0\}$, $a\in\bF(\alpha_1,\alpha_2,\ldots,\alpha_{n-1})[\alpha_n]$ and $\deg_{\alpha_n}(a)<\deg_{\alpha_n}(d) $. Note that each $\frac{a}{d^j}$ is in $V_{[d]_{G,j}}$.

Recall that the spread set $\Spr_{\sigma}(d,d)$ is a subgroup of $\bZ$ by Lemma \ref{lem-2-1}. Now we show the $c\sigma$-summability criteria for $\frac{a}{d^j}$ in $V_{[d]_{G,j}}$ with $\Spr_{\sigma}(d,d)=\{0\}$ and $\Spr_{\sigma}(d,d)=k\bZ$ for some $k\in\bN\setminus\{0\}$, respectively.

\begin{thm}\label{thm-3-1}
Let $f=\frac{a}{d^j}\in V_{[d]_{G,j}}$, where  $j\in\bN\setminus\{0\}$, $d\in\bF[\alpha_1,\alpha_2,\ldots,\alpha_{n}]$ with $d\nsim_{\sigma}\alpha_n$, $a\in\bF(\alpha_1,\alpha_2,\ldots,\alpha_{n-1})[\alpha_n]$ and $\deg_{\alpha_n}(a)<\deg_{\alpha_n}(d) $. If $\Spr_{\sigma}(d,d)=\{0\}$, then
$f$ is $c\sigma$-summable in  $\bF(\alpha_1,\ldots,\alpha_{n})$  if and only if $a=0$.
\end{thm}

\begin{proof}
If $a=0$, then $f=\Delta_{c\sigma}(0)$. For the necessary, assume $f=\Delta_{c\sigma}(g)$ with $g\in V_{[d]_{G,j}}$.

Write
\[g=\sum_{i={\ell}_0}^{{\ell}_1}\frac{b_i}{\sigma^i(d)^j},\]
where ${\ell}_0,{\ell}_1\in\bZ$,
${\ell}_1\geq{\ell}_0$ ,$b_{{\ell}_0}\neq0$, $b_{{\ell}_1}\neq0$ and $\deg_{\alpha_n}(b_i)<\deg_{\alpha_n}(d)$. Then
\begin{align*}
f=\Delta_{c\sigma}(g)=&\sum_{i={\ell}_0}^{{\ell}_1}\frac{c\sigma(b_i)}{\sigma^{i+1}(d)^j}-\sum_{i={\ell}_0}^{{\ell}_1}\frac{b_i}{\sigma^i(d)^j}\\
=&\frac{c\sigma(b_{{\ell}_0})}{\sigma^{{\ell}_0+1}(d)^j}-\frac{b_{{\ell}_0}}{\sigma^{l_0}(d)^j}+\frac{c\sigma(b_{{\ell}_0+1})}{\sigma^{{\ell}_0+2}(d)^j}-\frac{b_{{\ell}_0+1}}{\sigma^{{\ell}_0+1}(d)^j}+\cdots+\frac{c\sigma(b_{{\ell}_1})}{\sigma^{{\ell}_1+1}(d)^j}-\frac{b_{{\ell}_1}}{\sigma^{{\ell}_1}(d)^j}\\
=&-\frac{b_{{\ell}_0}}{\sigma^{{\ell}_0}(d)^j}+\left(\frac{c\sigma(b_{{\ell}_0})}{\sigma^{{\ell}_0+1}(d)^j}-\frac{b_{{\ell}_0+1}}{\sigma^{{\ell}_0+1}(d)^j}\right)+\cdots+\left(\frac{c\sigma(b_{{\ell}_1-1})}{\sigma^{{\ell}_1}(d)^j}-\frac{b_{{\ell}_1}}{\sigma^{{\ell}_1}(d)^j}\right)+\frac{c\sigma(b_{{\ell}_1})}{\sigma^{{\ell}_1+1}(d)^j}\\
=&\sum_{i={\ell}_0}^{{\ell}_1+1}\frac{\widetilde{b}_i}{\sigma^{i}(d)^j},
\end{align*}
where  $\widetilde{b}_{{\ell}_0}\neq0$,  $\widetilde{b}_{{\ell}_1}\neq0$, $\widetilde{b}_{{\ell}_0}=-b_{{\ell}_0}$, $\widetilde{b}_{{\ell}_1+1}=c\sigma(b_{{\ell}_1})$ and $\widetilde{b}_{{\ell}_i}=c\sigma(b_{{\ell}_i-1})-b_{i}$ for each ${\ell}_0<i\leq {\ell}_i$.
Also $f=\frac{a}{d^j}$, we have
$$\frac{a}{d^j}-\sum_{i={\ell}_0}^{{\ell}_1+1}\frac{\widetilde{b}_i}{\sigma^{i}(d)^j}=0.$$

Since  $\Spr_{\sigma}(d,d)=\{0\}$, $\sigma^{\ell_0}(d), \sigma^{\ell_0+1}(d),\cdots,\sigma^{\ell_1+1}(d)$ are pairwise coprime polynomials over $\bF(\alpha_1,\alpha_2,\ldots,\alpha_{n-1})$.
Then, by  $\widetilde{b}_{{\ell}_0}\neq0$ and $\widetilde{b}_{{\ell}_1}\neq0$,   there exist $u_0,u_1\in\bF(\alpha_1,\alpha_2,\ldots,\alpha_{n-1})\setminus\{0\}$ such that $\sigma^{\ell_0}(d)=u_0d$ and $\sigma^{\ell_1+1}(d)=u_1d$, respectively.
Moreover, $\sigma^{\ell_1+1-\ell_0}(d)=ud$, where $u=u_1/u_0\in\bF\setminus\{0\}$, since $\sigma$ does not change the total degree of a polynomial. Note that  $\ell_1+1-\ell_0>0$. That contradicts the fact that $\Spr_{\sigma}(d,d)=\{0\}$.
\end{proof}

Another criterion of $c\sigma$-summability with $n=1$, $\lambda_1=q$, $c=1$ and $q$ being not a root of unity  in Theorem \ref{thm-3-2} is given by Chen and Singer \citep{Chen-Singer-2012}.

\begin{thm}\label{thm-3-2}
Let $f=\frac{a}{d^j}\in V_{[d]_{G,j}}$ ,where  $j\in\bN\setminus\{0\}$, $d\in\bF[\alpha_1,\alpha_2,\ldots,\alpha_{n}]$ with $d\nsim_{\sigma}\alpha_n$, $a\in\bF(\alpha_1,\alpha_2,\ldots,\alpha_{n-1})[\alpha_n]$ and $\deg_{\alpha_n}(a)<\deg_{\alpha_n}(d) $. If $\Spr_{\sigma}(d,d)=k\bZ$ for some $k\in\bN\setminus\{0\}$ and $\sigma^k(d)=ud$ for some $u\in\bF\setminus\{0\}$, then
$f$ is $c\sigma$-summable in  $\bF(\alpha_1,\ldots,\alpha_{n})$ if and only if $a=\Delta_{\varepsilon\sigma^k}(b)$ for some $b\in\bF(\alpha_1,\ldots,\alpha_{n-1})[\alpha_n]$ with $\deg_{\alpha_n}(b)<\deg_{\alpha_n}(a)$ and $\varepsilon=c^ku^{-j}$.
\end{thm}

\begin{proof}
We firstly consider the sufficiency. Since $a=\Delta_{\varepsilon\sigma^k}(b)$,
\begin{align*}
\frac{a}{d^j}=\frac{c^ku^{-j}\sigma^k(b)-b}{d^j}=\frac{c^ku^{-j}\sigma^k(b)}{u^{-j}\sigma^k(d^j)}-\frac{b}{d^j}=((c\sigma)^k-1)\frac{b}{d^j}.
\end{align*}
Then $((c\sigma)^k-1)=(c\sigma-1)\cdot\sum_{i=1}^{k-1}(c\sigma)^{i}$ for all $k>0$, which implies $\frac{a}{d^j}=\Delta_{c\sigma}(g)$  with $g= \sum_{i=1}^{k-1}(c\sigma)^{i}\left(\frac{b}{d^j}\right)\in\bF(\alpha_1,\ldots,\alpha_{n})$.

For the necessary, assume $f=\Delta_{c\sigma}(g)$ with $g\in V_{[d]_{G,j}}$. Since $\Spr_{\sigma}(d,d)=k\bZ$ for some $k\in\bN\setminus\{0\}$,  we can write
\[g=\sum_{i=0}^{k-1}\frac{b_i}{\sigma^i(d)^j},\]
where
$b_i\in\bF(\alpha_1,\alpha_2,\ldots,\alpha_{n-1})[\alpha_n]$ and $\deg_{\alpha_n}(b_i)<\deg_{\alpha_n}(d)$. Then

\begin{align*}
f=\Delta_{c\sigma}(g)=&\sum_{i=0}^{k-1}\frac{c\sigma(b_i)}{\sigma^{i+1}(d)^j}-\sum_{i=0}^{k-1}\frac{b_i}{\sigma^i(d)^j}\\
=&\frac{c\sigma(b_{0})}{\sigma(d)^j}-\frac{b_{0}}{d^j}+\frac{c\sigma(b_{1})}{\sigma^{2}(d)^j}-\frac{b_{1}}{\sigma(d)^j}+\cdots+\frac{c\sigma(b_{k-1})}{\sigma^{k}(d)^j}-\frac{b_{k-1}}{\sigma^{k-1}(d)^j}\\
=&-\frac{b_{0}}{d^j}+\left(\frac{c\sigma(b_{0})}{\sigma(d)^j}-\frac{b_{1}}{\sigma(d)^j}\right)+\cdots+\left(\frac{c\sigma(b_{k-2})}{\sigma^{k-1}(d)^j}-\frac{b_{k-1}}{\sigma^{k-1}(d)^j}\right)+\frac{c\sigma(b_{k-1})}{\sigma^{k}(d)^j}\\
=&\left(\frac{cu^{-j}\sigma(b_{k-1})}{d^j}-\frac{b_{0}}{d^j}\right)+\left(\frac{c\sigma(b_{0})}{\sigma(d)^j}-\frac{b_{1}}{\sigma(d)^j}\right)+\cdots+\left(\frac{c\sigma(b_{k-2})}{\sigma^{k-1}(d)^j}-\frac{b_{k-1}}{\sigma^{k-1}(d)^j}\right)\\
=&\sum_{i=0}^{k-1}\frac{\widetilde{b}_i}{\sigma^{i}(d)^j},
\end{align*}
where $\widetilde{b}_{0}=cu^{-j}\sigma(b_{k-1})-b_{0}$, $\widetilde{b}_{i}=c\sigma(b_{i-1})-b_i$ for each $1\leq i\leq k-1$.
Since  $\Spr_{\sigma}(d,d)=k\bZ$, $d$, $\sigma(d), \cdots, \sigma^{k-1}(d)$ are pairwise coprime polynomials over $\bF(\alpha_1,\alpha_2,\ldots,\alpha_{n-1})$.

Then, by $f=\frac{a}{d^j}$, we have
\[\left\{
  \begin{array}{ll}
    a+b_0-cu^{-j}\sigma(b_{k-1})=0 \\
    b_1=c\sigma(b_0) \\
   b_2=c\sigma(b_1) \\
    \vdots\\
b_{k-1}=c\sigma(b_k)
  \end{array}
\right.,\]
which implies
\[a=c^ku^{-j}\sigma^k(b_0)-b_0=\Delta_{\varepsilon\sigma^k}(b).\]
\end{proof}

For the $c\sigma$-summability of a simple fraction $\frac{a}{d^j}$ in $ V_{[d]_{G,j}}$,  by Theorem \ref{thm-3-2}, it suffices to consider the $\varepsilon\sigma^k$-summability of $a$ in $\bF(\alpha_1,\ldots,\alpha_{n-1})[\alpha_n]$.

\begin{thm}\label{thm-3-3}
For $a\in\bF(\alpha_1,\ldots,\alpha_{n-1})[\alpha_n]$, write $a=\sum_{i=0}^{\ell}a_i\alpha_n^i$ with $a_i\in\bF(\alpha_1,\ldots,\alpha_{n-1})$, then
$a=\Delta_{\varepsilon\sigma^k}(g)$ with $\deg_{\alpha_n}(g)<\deg_{\alpha_n}(a)$ in $\bF(\alpha_1,\ldots,\alpha_{n-1})[\alpha_n]$ if and only if
each $a_i$ is $\varepsilon'\sigma^k$-summable in $\bF(\alpha_1,\ldots,\alpha_{n-1})$ with $\varepsilon'=\varepsilon\lambda_n^{ki}$.
\end{thm}

\begin{proof}
The proof is entirely analogous to that of Lemma \ref{lem-3-2}, and is omitted.

%
\end{proof}

Another criterion of $c\sigma$-summability with $n=1$, $\lambda_1=q$,  $c=1$, and $q$ being a root of unity in Theorem \ref{thm-3-3} is given by Chen and Singer \citep{Chen-Singer-2014}.

According to Theorems \ref{thm-3-2} and \ref{thm-3-3}, we reduce the summability problem in $V_{[d]_{G,j}}$ with $n$ variables to the summability problem with $n-1$ variables, hence the summability problem can be solved by the reduction.
Based on these, we present an algorithm for computing one solution $g\in\bF(\alpha_1,\alpha_2
\ldots,\alpha_n)$ of the equation \eqref{eq-3-3}:

\begin{breakablealgorithm}\label{alg:exist}
\caption{IsSummable ($f, [\alpha_1,\ldots,\alpha_{n}], c\sigma$)}
\begin{algorithmic}[htb]
	\textbf{Input:} a rational function $f\in\bF(\alpha_1,\alpha_2
\ldots,\alpha_n)$, $c\in\bF\setminus\{0\}$, and an automorphism $\sigma$ on $\bF(\alpha_1,\alpha_2
\ldots,\alpha_n)$ such that $\sigma(\alpha_i)=\lambda_i\alpha_i$  with $\lambda_i\in\bF\setminus\{0\}$\\
\textbf{Output:} $g\in\bF(\alpha_1,\alpha_2
\ldots,\alpha_n)$, if $f=c\sigma(g)-g$ for some $g\in\bF(\alpha_1,\alpha_2
\ldots,\alpha_n)$; False, otherwise.

    \step 11 using the spread computation and partial fraction decomposition, decompose $f$ into
\[
f=\sum_{i\in\Lambda}p_i\alpha_n^i+\sum_{j=1}^J\sum_{[d]_{G}}f_{[d]_{G,j}},
\]
where  $\Lambda\subset\bZ$, $J\in\bN\setminus\{0\}$, $p_i\in\bF(\alpha_1,\ldots,\alpha_{n-1})\setminus\{0\}$,  $p_i\alpha_n^i$ are in distinct $V_i$ space,  and $f_{[d]_{G,j}}$ are in distinct $V_{[d]_{G,j}}$ space.
    \step 21 apply the reduction to each component $f_{[d]_{G,j}}$ such that
    \[
    f=\Delta_{c\sigma}(g)+r, \text{ with } r=\sum_{i\in\Lambda}p_i\alpha_n^i+\sum_{i=1}^I\sum_{j=1}^{J_i}\frac{a_{i,j}}{d_{i}^{j}},
    \]
where $I,J_i\in\bN\setminus\{0\}$, and each $\frac{a_{i,j}}{d_{i}^{j}}$ is the reminder of $f_{[d_i]_{G,j}}$ in Lemma \ref{lem-3-4}.
    \step 31 if $\sum_{i\in\Lambda}p_i\alpha_n^i=0$ and $r=0$, {\em\bf return} $g$;
    \step 41 if $\sum_{i\in\Lambda}p_i\alpha_n^i\neq0$, then for $i\in\Lambda$, do
     \step 53 if $n>1$,
     \step { }4 IsSummable ($p_i, [\alpha_1,\ldots,\alpha_{n-1}], c\lambda_n^i\sigma$).
     \step {6}4if $p_i$ is $c\lambda_n^i\sigma$-summable, let $q_i$ be such that $p_i=\Delta_{ c\lambda_n^i\sigma}(q_i)$;
     \step { }4 otherwise {\em\bf return} False.
     \step 73 else ($n=1$)
     \step {8}4 if $\lambda_1^i\neq\frac{1}{c}$, let $q_i=\frac{p_i}{c\lambda_1^i-1}$;
     \step { }4 otherwise {\em\bf return} False.
   \step{9}3 update $g=g+q_i\alpha_n^i$.
    \step {10}2 for $i$ from 1 to $I$, do
    \step { }3 compute $k_i\in\bN$ and $u_i\in\bF$ such that $\Spr_{\sigma}(d_i, d_i)=k_i\bZ$ and $\sigma^{k_i}(d_i)=u_id_i$.
    \step {11}3 for $j$ from 1 to $J_i$, do
     \step{12}4 if $\Spr_{\sigma}(d_i,d_i)=\{0\}$, then
    \step{ }5 if $a_{i,j}\neq0$, {\em\bf return} False;
    \step{13}4 else
    \step{ }5
     write $a_{i,j}=\sum_{\ell\in L_{i,j}}a_{i,j,\ell}\alpha_n^{\ell}$, where $a_{i,j,\ell}\in\bF(\alpha_1,\alpha_2,\ldots,\alpha_{n-1})\setminus\{0\}$ and $L_{i,j}\subset\bN$.
    \step{ }5  let $\varepsilon_{i,j}=c^{k_i}u_i^{-j}$.
    \step {14}5 for $\ell\in L_{i,j}$, do
    \step {15}6 if $n>1$,
    \step { }7 IsSummable($a_{i,j,\ell}, [\alpha_1,\alpha_2,\ldots,\alpha_{n-1}], \varepsilon_{i,j}\lambda_n^{\ell k_i}\sigma^{k_i}$).
    \step{16}7 if $a_{i,j,\ell}$ is $\varepsilon_{i,j}\lambda_n^{\ell k_i}\sigma^{k_i}$-summable, let $b_{i,j,\ell}$ be such that $a_{i,j,\ell}=\Delta_{\varepsilon_{i,j}\lambda_n^{\ell k_i}\sigma^{k_i}}(b_{i,j,\ell})$;
    \step{ }7 otherwise {\em\bf return} False.
    \step {17}6 else ($n=1$)
    \step {18}7 if $\lambda_i^{\ell k_i}\neq1/\varepsilon_{i,j,}$, let $b_{i,j,   \ell}=\frac{a_{i,j,\ell}}{\varepsilon_{i,j}\lambda_1^{\ell k_i}-1}$;
    \step{ }7 otherwise {\em\bf return} False;
    \step{19}6 update $g=g+\sum_{m=1}^{k_i-1}c^m\sigma^{m}\left(\frac{b_{i,j,\ell}\alpha_n^{\ell}}{d_i^j}\right)$.
     \step {21}2 {\em\bf return} $g$
\end{algorithmic}
\end{breakablealgorithm}

In Algorithm  \ref{alg:exist}, ``False'' represents $f$ is not $c\sigma$-summable in $\bF(\alpha_1,\alpha_2,\ldots,\alpha_n)$.


\section{The rational solution set}\label{sec-4}
In this section, let $(\bF(\alpha_1,\alpha_2,\ldots,\alpha_n),\sigma)$ be the multivariate difference field, which is defined by \eqref{eq-2-0},
$c\in\bF\setminus\{0\}$, $f\in\bF(\alpha_1,\alpha_2,\ldots,\alpha_n)$, we construct all of rational solutions $g\in\bF(\alpha_1,\alpha_2,\ldots,\alpha_n)$ of the difference equation
\begin{align}\label{eq-4-1}
c\sigma(g)-g=f.
\end{align}
Let $\Sol_{c,f}$ denote \emph{the set of the rational solutions} of \eqref{eq-4-1}.

First, we show a relationship between Sol$_{c,f}$ and the constant field $\Const_{\sigma}\bF(\alpha_1,\alpha_2,\ldots,\alpha_n)$ of the multivariate difference field $(\bF(\alpha_1,\alpha_2,\ldots,\alpha_{n}),\sigma)$.

\begin{thm}
Let $g^*$ be a rational solution of the equation \eqref{eq-4-1}.

$(1)$ If the homogeneous difference equation $c\sigma(g)-g=0$ has only zero solution, then
\[
\Sol_{c,f}=\{g^*\}.
\]

$(2)$ If the homogeneous difference equation $c\sigma(g)-g=0$ has a nonzero solution $\widetilde{g}$, then
\begin{align*}
\text{Sol}_{c,f}=\{F\widetilde{g}+g^*\mid F\in \Const_{\sigma}\bF(\alpha_1,\alpha_2,\ldots,\alpha_n)\}.
\end{align*}
\end{thm}

\begin{proof}
$(1)$ Without lose of generality, assume $g\in\text{Sol}_{c,f}$ with $g\neq g^*$, then
$c\sigma(g-g^*)-(g-g^*)=(c\sigma(g)-g)+(c\sigma(g^*-g^*))=f-f=0$. So $g-g^*$ is a nonzero solution of the equation $c\sigma g-g=0$, which is a contradiction.

$(2)$ If $F\in \Const_{\sigma}\bF(\alpha_1,\alpha_2,\ldots,\alpha_n)$, it can be directly checked that $F\widetilde{g}+g^*$ is a solution of  the equation \eqref{eq-4-1}.

For the necessary, take an arbitrary rational solution $g$ of $\eqref{eq-4-1}$, then
\begin{align*}
\sigma\left(\frac{g-g^*}{\widetilde{g}}\right)\bigg{/}\left(\frac{g-g^*}{\widetilde{g}}\right)=\frac{\sigma(g-g^*)}{g-g^*}\cdot\frac{\widetilde{g}}{\sigma \widetilde{g}}=c\cdot\frac{1}{c}=1.
\end{align*}
That implies $\frac{g-g^*}{\widetilde{g}}\in\Const_{\sigma}\bF(\alpha_1,\alpha_2,\ldots,\alpha_n)$.
By the arbitrariness of $g$, we get the result.
\end{proof}

We could get one rational solution $g^*$ of the difference equation \eqref{eq-4-1} by Algorithm \ref{alg:exist}, if it exists. Next, we only need find a nonzero rational solution $\widetilde{g}$ of the difference equation
\begin{align}\label{eq-4-2}
c\sigma(g)-g=0,
 \end{align}
where $c\in\bF\setminus\{0\}$.

\begin{thm}
Let $c\in\bF\setminus\{0\}$. If $c\notin\{\lambda_1^{i_1}\lambda_2^{i_2}\ldots\lambda_n^{i_n}| i_1,i_2\ldots,i_n\in\bZ\}$, then the equation \eqref{eq-4-2} has only zero solution.
\end{thm}

\begin{proof}
Suppose $\widetilde{g}=p/q$ is a nonzero rational solution of the equation \eqref{eq-4-2}, where $p,q$ $\in$ $\bF[\alpha_1,\alpha_2,\ldots,\alpha_n]\setminus\{0\}$ are two coprime polynomials. Then
\begin{align*}
\frac{\sigma p}{p}\cdot\frac{q}{\sigma q}&=1/c,
\end{align*}
and there exist $t_1,t_2\in\bF\setminus\{0\}$ such that
\begin{align}\label{eq-4-3}
\frac{\sigma p}{p}=t_1,\quad\frac{\sigma q}{q}=t_2, \text{ and }t_1/t_2=1/c.
\end{align}

Without lose of generality, we may further suppose $p, q$ are monic and \begin{align*}
p=\boldsymbol{\alpha}^{{\bf m}_p}+\sum_{\boldsymbol{\alpha}^{\bf i}<\boldsymbol{\alpha}^{{\bf m}_p}}p_{\bf i}\boldsymbol{\alpha}^{\bf i},\quad q=\boldsymbol{\alpha}^{{\bf m}_q}+\sum_{\boldsymbol{\alpha}^{\bf j}<\boldsymbol{\alpha}^{{\bf m}_q}}q_{\bf j}\boldsymbol{\alpha}^{\bf j},
\end{align*}
where $p_{\bf i}, q_{\bf i}\in\bF$, ${{\bf m}_p}=(m_{1p},m_{2p},\ldots,m_{np}), {{\bf m}_q}=(m_{1q},m_{2q},\ldots,m_{np}), {\bf i}=(i_1,i_2,\ldots,i_n), {\bf j}=(j_1,j_2,\ldots,j_n)\in\bN^n$, ${\boldsymbol{\alpha}}^{\bf{m}_p}=\alpha_1^{m_{1p}}\alpha_2^{m_{2p}}\ldots\alpha_n^{m_{np}}$, ${\boldsymbol{\alpha}}^{\bf{m}_q}=\alpha_1^{m_{1q}}\alpha_2^{m_{2q}}\ldots\alpha_n^{m_{nq}}$,
${\boldsymbol{\alpha}}^{\bf i}={\alpha_1}^{i_1}{\alpha_2}^{i_2}\ldots{\alpha_n}^{i_n}$ and ${\boldsymbol{\alpha}}^{\bf j}={\alpha_1}^{j_1}{\alpha_2}^{j_2}\ldots{\alpha_n}^{j_n}$.

Then, by \eqref{eq-4-3} and Lemma \ref{lem-2-2}, we have
\[\sigma(\boldsymbol{\alpha}^{{\bf m}_p})=\boldsymbol{\lambda}^{{\bf m}_p}\cdot\boldsymbol{\alpha}^{{\bf m}_p}=t_1\boldsymbol{\alpha}^{{\bf m}_p},\quad \sigma(\boldsymbol{\alpha}^{{\bf m}_q})=\boldsymbol{\lambda}^{{\bf m}_q}\cdot\boldsymbol{\alpha}^{{\bf m}_q}=t_2\boldsymbol{\alpha}^{{\bf m}_q},\]
where $\boldsymbol{\lambda}^{{\bf m}_p}={\lambda_1}^{m_{1p}}{\lambda_2}^{m_{2p}}\ldots{\lambda_n}^{m_{np}}$
and $\boldsymbol{\lambda}^{{\bf m}_q}={\lambda_1}^{m_{1q}}{\lambda_2}^{m_{2q}}\ldots{\lambda_n}^{m_{nq}}$.
So
\[\boldsymbol{\lambda}^{{\bf m}_q}/\boldsymbol{\lambda}^{{\bf m}_p}=t_2/t_1=c.\]
That contradicts the fact that $c\notin\{\lambda_1^{i_1}\lambda_2^{i_2}\ldots\lambda_n^{i_n}| i_1,i_2\ldots,i_n\in\bZ\}$.
\end{proof}

\begin{thm}
Let $c\in\bF\setminus\{0\}$. If there exist $i_1,i_2,\ldots, i_n\in\bZ$ such that $\lambda_1^{i_1}\lambda_2^{i_2}\cdots\lambda_n^{i_n}=c^{-1}$, then
$g\widetilde{}=\alpha_1^{i_1}\alpha_2^{i_2}\cdots\alpha_n^{i_n}$ is a rational solution of the equation \eqref{eq-4-2}.
\end{thm}
\begin{proof}
Substituting $\widetilde{g}=\alpha_1^{i_1}\alpha_2^{i_2}\cdots\alpha_n^{i_n}$ into the left hand side of the equation \eqref{eq-4-2}, we have
\begin{align*}
c\sigma \widetilde{g}-\widetilde{g}&=c\sigma(\alpha_1^{i_1}\alpha_2^{i_2}\cdots\alpha_n^{i_n})-\alpha_1^{i_1}\alpha_2^{i_2}\cdots\alpha_n^{i_n}\\
&=c\cdot\lambda_1^{i_1}\lambda_2^{i_2}\cdots\lambda_n^{i_n}\cdot\alpha_1^{i_1}\alpha_2^{i_2}\cdots\alpha_n^{i_n}-\alpha_1^{i_1}\alpha_2^{i_2}\cdots\alpha_n^{i_n}\\
&=c\cdot c^{-1}\cdot\alpha_1^{i_1}\alpha_2^{i_2}\cdots\alpha_n^{i_n}-\alpha_1^{i_1}\alpha_2^{i_2}\cdots\alpha_n^{i_n}\\
&=0.
\end{align*}
\end{proof}

Note that, for given $\lambda_1, \lambda_2, \ldots, \lambda_n, c$, we could find $i_1,\ldots,i_n\in\bZ$ such that $\lambda_1^{i_1}\lambda_2^{i_2}\cdots\lambda_n^{i_n}=c^{-1}$ by using the algorithm about the exponent lattice in \citep{Manuel-2023}.

\smallskip{}
\noindent \textbf{Acknowledgement.}
The authors would like to thank Manuel Kauers for discussions about the exponent lattice problem of algebraic numbers. L.\ Du was supported by the Austrian FWF grant P31571--N32.



%


\end{document}